\documentclass[10pt]{amsart}
\usepackage{amsmath}
\usepackage{amsfonts}
\usepackage{amssymb}
\usepackage{amsthm}
\usepackage{url}
\usepackage{tikz-cd}
\usepackage{dsfont}
\usepackage{graphicx}
\usepackage{caption}
\usepackage{subcaption}
\usepackage{comment} 
\usepackage{stmaryrd}
\usepackage{hyperref}
\usepackage{todonotes}
\usepackage{color}
\usepackage{enumerate}

\newcommand{\git}{\mathbin{
  \mathchoice{/\mkern-6mu/}
    {/\mkern-6mu/}
    {/\mkern-5mu/}
    {/\mkern-5mu/}}}

\numberwithin{equation}{section}

\newtheorem{proposition}{Proposition}[section]
\newtheorem{lemma}[proposition]{Lemma}
\newtheorem{theorem}[proposition]{Theorem}
\newtheorem{corollary}[proposition]{Corollary}

\theoremstyle{definition}
\newtheorem{remark}[proposition]{Remark}
\newtheorem{definition}[proposition]{Definition}

\DeclareMathOperator{\Id}{Id}
\DeclareMathOperator{\tr}{tr}

\DeclareMathOperator{\Aut}{Aut}
\DeclareMathOperator{\Ric}{Ric}
\DeclareMathOperator{\Grass}{Grass}
\DeclareMathOperator{\Ch}{Ch}
\DeclareMathOperator{\SL}{SL}
\DeclareMathOperator{\Proj}{Proj}
\DeclareMathOperator{\V}{V}

\DeclareMathOperator{\Lie}{Lie}

\renewcommand{\phi}{\varphi}

\newcommand{\C}{\mathbb{C}}

\newcommand{\pr}{\mathbb{P}}
\renewcommand{\epsilon}{\varepsilon}
\newcommand{\scH}{\mathcal{H}}
\newcommand{\D}{\mathcal{D}}
\newcommand{\M}{\mathcal{M}}

\newcommand{\scO}{\mathcal{O}}

\newcommand{\ddb}{i\partial \bar\partial}
\newcommand{\scL}{\mathcal{L}}
\newcommand{\mfh}{\mathfrak{h}}
\newcommand{\mft}{\mathfrak{t}}

\newcommand{\mfp}{\mathfrak{p}}
\newcommand{\mfl}{\mathfrak{l}}

\newcommand{\scB}{\mathcal{B}}
\newcommand{\mfb}{\mathfrak{b}}

\newcommand{\scV}{\mathcal{V}}

\newcommand{\scU}{\mathcal{U}}

\title[Extremal metrics on fibrations]{Extremal metrics on fibrations}

\author[Ruadha\'i Dervan and Lars Martin Sektnan]{Ruadha\'i Dervan and Lars Martin Sektnan}
\address{Ruadha\'i Dervan, DPMMS, Centre for Mathematical Sciences, Wilberforce Road, Cambridge CB3 0WB, United Kingdom and Centre de math\'ematiques Laurent-Schwartz, \'Ecole Polytechnique, Cour Vaneau, 91120 Palaiseau, France}
\email{R.Dervan@dpmms.cam.ac.uk}

\address{Lars Martin Sektnan, D\'epartement de math\'ematiques, Universit\'e du Qu\'ebec \`a Montr\'eal, Case postale 8888, succursale
centre-ville, Montr\'eal (Qu\'ebec), H3C 3P8, Canada}
\email{lars.sektnan@cirget.ca}

\begin{document}

\maketitle

\begin{abstract} Consider a fibred compact K\"ahler manifold $X$ endowed with a relatively ample line bundle, such that each fibre admits a constant scalar curvature K\"ahler metric and has discrete automorphism group. Assuming the base of the fibration admits a twisted extremal metric where the twisting form is a certain Weil-Petersson type metric, we prove that $X$ admits an extremal metric for polarisations making the fibres small. Thus $X$ admits a constant scalar curvature K\"ahler metric if and only if the Futaki invariant vanishes. This extends a result of Fine, who proved this result when the base admits no continuous automorphisms. 

As  consequences of our techniques, we obtain analogues for maps of various fundamental results for varieties: if a map admits a twisted constant scalar curvature K\"ahler metric metric, then its automorphism group is reductive; a twisted extremal metric is invariant under a maximal compact subgroup of the automorphism group of the map;  there is a geometric interpretation for uniqueness of twisted extremal metrics on maps. 
\end{abstract}

\maketitle

\section{Introduction}
A basic question in complex geometry is whether or not a given K\"ahler manifold admits an extremal K\"ahler metric. A special case of an extremal K\"ahler metric is one that has constant scalar curvature (henceforth cscK), as the extremal condition simply means that the $((1,0)$-part of the) gradient of the scalar curvature is a holomorphic vector field. One motivation for this question is the link with moduli theory: it is expected that one can form well-behaved moduli spaces of such manifolds. The goal of the present work is to give a new construction of extremal K\"ahler manifolds, using some ideas from moduli theory.

To motivate our work, we first review some older constructions of extremal K\"ahler metrics. The first general construction of such metrics is due to Hong \cite{hong1}, who considered fibrations $\pi: \pr(E) \to B$, where $\pr(E)$ denotes the projectivisation of a vector bundle $E$ on a compact complex manifold $B$. Suppose $B$ admits a line bundle $L$ with a cscK metric, and suppose in addition that $E$ admits a Hermite-Einstein metric. Assume moreover that the data involved does not admit continuous automorphisms, i.e. $\Aut(X,L)$ is discrete and $E$ is indecomposable. Then Hong proves that $\pr(E)$ admits a cscK metric in the class $rL+\scO_{\pr(E)}(1)$ for $r\gg 0$ \cite{hong1}, where we have used additive notation for the tensor product of line bundles, and suppressed pullbacks. In later work Hong relaxes the assumption on the automorphisms, showing that provided $E$ is $\Aut(X,L)$-invariant then the class $rL+\scO_{\pr(E)}(1)$ still admits a cscK metric provided the so-called Futaki invariant vanishes on $\pr(E)$ \cite{hong2}. Most recently Hong has proved that the analogous result holds even if $E$ is \emph{not} $\Aut(X,L)$-invariant, however in this case $E$ is required to satisfy a certain finite-dimensional stability condition \cite{hong3}. The latter result was extended to the case that $B$ instead admits an extremal metric by Lu-Seyyedali \cite{lu-seyyedali}, and related results have also been proven in the case that $E$ is no longer indecomposable \cite{bronnle,acgt11}. 

Projective bundles as above form one extreme of the types of fibrations. The other extreme is a fibration $\pi: X \to B$ such that each fibre is a smooth projective manifold with \emph{discrete} automorphism group, but with possibly varying complex structure (thus the map $\pi$ is a holomorphic submersion). This is the situation considered by Fine \cite{fine1}. Assume that $X$ admits a relatively ample line bundle $H$ such that the fibre $(X_b,H_b)$ has discrete automorphism group and admits a cscK metric for all $b\in B$, and that $B$ admits an ample line bundle $L$ with a \emph{unique} twisted cscK metric, i.e. one satisfying $S(\omega) - \Lambda_{\omega}\alpha = const$, for $\alpha$ a closed $(1,1)$-form which is determined by the geometry of the fibration. Then Fine proves that $X$ admits a cscK metric in the class $rL+H$ for $r \gg 0$. In fact Fine's proof contains a gap in the case the base of the fibration has complex dimension at least two, in that the uniqueness assumption he makes is not enough guarantee that the linear operator he utilises is invertible, as we discuss further in Remark \ref{rmk:second-uniqueness}. 

Our main result is the following:

\begin{theorem}\label{intromainthm}
Suppose $(X,H) \to (B,L)$ is a fibration, with $B$ and $X$ compact, such that each fibre $(X_b,H_b)$ has discrete automorphism group and admits a cscK metric for all $b\in B$. Suppose in addition that $(B,L)$ admits a twisted extremal metric with twist $\alpha$ a closed $(1,1)$-form as above, and that the extremal vector field on $B$ lifts to $X$. Then $(X,rL+H)$ admits an extremal metric for all $r\gg 0$. 
\end{theorem}

Here a twisted extremal metric satisfies the condition that the gradient of $S(\omega) - \Lambda_{\omega} \alpha$ is a holomorphic vector field, called the extremal vector field. In the cscK case we obtain:

\begin{corollary}\label{intromaincor}
With all notation as above, $(X,rL+H)$ admits a cscK metric for all $r\gg 0$ if and only if the Futaki invariant vanishes. \end{corollary}
We remark, though, that it may very well be that the Futaki invariant does not vanish for any K\"ahler classes on $X$, see e.g. \cite[Section 3]{acgt08}. Thus it seems important to allow extremal metrics on $X$ in general.

In particular we are \emph{not} assuming $\omega$ is a \emph{unique} solution of the twisted extremal equation. We shall see in Remark \ref{rmk:second-uniqueness} that the uniqueness assumption in Fine's result is analogous to Hong's assumption that the base has no automorphisms in Hong's first result \cite{hong1}, and thus the corollary above is analogous to Hong's second result \cite{hong2}. Surprisingly, while Hong had to assume that the vector bundle $E$ he considers is $\Aut(X,L)$-invariant in \cite{hong2}, we do not need to make any such an assumption on the fibration $X\to B$. Thus our main result proves also the analogous result of Hong's most recent work described above \cite{hong3}, and also that of Lu-Seyyedali \cite{lu-seyyedali}. Moreover, our result fixes the gap in Fine's work in the case that the base of the fibration has dimension at least two, and also shows that his uniqueness assumption is essentially unnecessary.

The starting point of our result is the perspective of \cite{dervan-ross}, that twisted cscK and extremal metrics are best understood when $\alpha$ is the pullback of some K\"ahler form from a map $p: B \to \M$. Indeed in this case, through \cite{dervan-ross,twisted}, the existence of a twisted cscK metric is conjecturally equivalent to a notion of K-stability for the map $p: B\to \M$ (with one direction of this conjecture proven in \cite{twisted}). To view $\alpha$ in this way, we first observe that $\alpha$ can be identified with a Weil-Petersson metric induced from the cscK metrics on the fibres $(X_b,L_b)$, as introduced by Fujiki-Schumacher \cite{fujiki-schumacher}.

In order to see $\alpha$ as the pullback of a \emph{K\"ahler} form through a morphism, we wish to use a moduli space which parametrises the fibres. As this is rather fundamental to our approach, we do this in two ways. The first is to directly appeal to the construction of Fujiki-Schumacher of a moduli space of cscK manifolds \cite{fujiki-schumacher}, on which they show the Weil-Petersson metric is K\"ahler with $\alpha \in c_1(L_{CM})$, where $L_{CM}$ is the CM-line bundle of \cite{fujiki-schumacher,tian-cm}. The second is a more direct construction of the moduli space via algebro-geometric methods, by using Donaldson's fundamental result that a projective manifold which admits a cscK metric, and with discrete automorphism group, is asymptotically Chow stable \cite{donaldson-projective-embeddings}. We can then construct an appropriate moduli space as a subvariety of a quotient of a moduli space of certain Chow stable varieties, on which the Weil-Petersson metric is shown to have the desired positivity properties.

The next step in the proof is to understand the properties of the linearisation of the twisted extremal operator $\scL_{\alpha}$ at a twisted extremal metric. This is the key step in forming an approximate solution to the extremal equation on $X$ itself. We show that when $\alpha$ is the pullback of a K\"ahler form through $p$ as above, the kernel of $\scL_{\alpha}$ can be identified with the vector fields whose flows induce automorphisms of the map $p: B\to \M$. Since we have realised $\M$ as a moduli space parametrising the fibres, using some moduli theory we show  that these automorphisms are precisely those which lift to $X$ (this is clearest working with the moduli space as a quotient of a Chow variety), and therefore that all holomorphic vector fields on $X$ can be identified with lifts of holomorphic vector fields on $B$ preserving $p$. The proof of this crucially uses that the fibres have discrete automorphism group: as one can see in the example of projective bundles and which we explain in Remark \ref{lifting-remark-bundle}, the analogous result does not hold for arbitrary fibrations. While $\scL_{\alpha}$ is not invertible in our situation, by searching for extremal metrics on $X$ instead of cscK metrics we show that one can adapt Fine's technique in the automorphism-free case to produce an approximately extremal metric. An implicit function theorem argument then gives a genuine extremal metric, proving Theorem \ref{intromainthm}. Corollary \ref{intromaincor} is then a simple consequence of the fact that an extremal metric is cscK if and only if the Futaki invariant vanishes.

\subsection{The geomtery of maps} As mentioned above, twisted cscK metrics are conjecturally linked to stability notions for maps between polarised varieties, and hence to the moduli theory of maps \cite{dervan-ross,twisted}. One of the fundamental aspects of the relationship between (genuine) cscK metrics and stability is the link between the scalar curvature operator and the geometry of the polarised variety $(X,H)$ (here $H$ is an ample line bundle on $X$). The basic result along these lines is that the kernel of the Lichnerowicz operator (the linearisation of the scalar curvature at a cscK metric) can be identified with the Lie algebra of the automorphism group $\Aut(X,H)$. We shall show in Section \ref{sec:lifting} that an analogous result holds for maps, namely the kernel of the ``twisted'' Lichnerowicz operator on a map $p: (B,L) \to (M,L_M)$ can be identified with the Lie algebra $\mfp$ of the automorphism group of the \emph{map} $p$. Our first concrete consequence of this identification is the following analogue of Matsushima's theorem:

\begin{theorem}\label{intro:red} Suppose the map $p$ admits a twisted cscK metric. Then the Lie algebra $\mfp$ is reductive. \end{theorem}

We further give a detailed description of the automorphism group on a map admitting a twisted extremal metric. Similar ideas give the following.

\begin{theorem}\label{intro:unique} A twisted extremal metric is unique up to the action of the automorphism group of the map. Moreover, a twisted extremal metric is invariant under a maximal compact subgroup the same automorphism group.\end{theorem}

For the first statement our main contribution is the geometric interpretation rather than the uniqueness result \emph{per se}. Deep work of Berman-Berndtsson \cite{berman-berndtsson} proves uniqueness of extremal metrics, and uniqueness of twisted extremal metrics when the twisting form $\alpha$ is positive (so from our perspective $p$ is an embedding). In fact their techniques apply to a broader class of equations without change, provided one understands the Hessian of the appropriate log norm functional. For us this functional is the twisted Mabuchi functional, and we show the Hessian of this functional degenerates precisely along the automorphisms of the map $p$, leading to the above statement. In particular, it follows (see Remark \ref{rmk:second-uniqueness}) that the uniqueness assumption in Fine's work \cite{fine1} is equivalent to demanding that the total space of the fibration $(X,rL+H)$ admit no continuous automorphisms. The automorphism invariance of a twisted extremal metric is a consequence of the structure theory of the automorphism group of the map that we develop, in a similar way to Calabi's results for extremal metrics. We hope these results demonstrate clearly that the natural geometric setting for the study of twisted cscK metrics is on maps between polarised varieties. 

Twisted extremal metrics play an analogous role to extremal metrics for polarised varieties, for example they are critical points of the twisted Calabi functional. It is thus natural to expect that there is a stability notion associated to the existence of twisted extremal metrics, extending the notion introduced by Sz\'ekelyhidi for extremal metrics \cite{szekelyhidi-extremal} and building on his finite dimensional stability interpretation for critical points of the norm squared of a moment map. The link between critical points of moment maps and moduli theory was elucidated by Atiyah-Bott \cite{atiyah-bott}, and thus it is natural to expect twisted extremal metrics similarly play an important role in the moduli theory of maps. We refer to \cite{stoppa-szekelyhidi,relative} for results linking extremal metrics to stability notions.

\subsection{Moduli} As remarked above, we give a construction the moduli space parametrising the fibres as a quotient of a certain Chow variety. This technique allows us to give new, simple proofs (though perhaps known to experts) of two deep results of Fujiki-Schumacher \cite{fujiki-schumacher} regarding the properties of the moduli space cscK manifolds: such a moduli space is automatically separated, and any proper subspace of such a moduli space is projective. 

\subsection{Examples}

While our main interest in Theorem \ref{intromainthm} is in the interplay between the existence theory for extremal metrics and the moduli theory of the fibres, it is an important problem to give new explicit examples of extremal K\"ahler manifolds. One way of constructing such examples from our main result, at least modulo a result regarding the existence of twisted extremal metrics on blowups, is as follows. We begin with a fibration $\pi: (X,H)\to (B,L)$ with $B$ a Riemann surface of genus two and such that all fibres are Riemann surfaces of genus at least two. Such examples were considered by Fine \cite{fine1}, who showed that such $B$ admit a solution of the twisted cscK equation for all $L$ for the relevant twisting induced by the moduli map. Consider the product $(Y = B\times \pr^1,L_Y = p_1^*L\otimes p_2^*\scO(1))$, and the pullback family $\pi': (X',H') = (X,H)\times_B (B,H) \to (Y,L_Y)$. Then $(Y,L_Y)$ admits a twisted cscK metric with respect to the induced Weil-Petersson metric on $Y$. A well-known result of Arezzo-Pacard-Singer states that if one blows-up an automorphism invariant point on an extremal K\"ahler manifold $(Z,L_Z)$, then the blow-up $(Bl_pZ,L_Z - \epsilon E)$ admits an extremal metric with $E$ the exceptional divisor and for $\epsilon$ sufficiently small \cite{APS}. Given our linearisation results, it is natural to expect that the same is true in the twisted setting (with the automorphism group of the K\"ahler manifold replaced by the automorphism group of the map), and hence blowing-up an automorphism invariant point on our $Y$ should produce a twisted extremal metric on $(Bl_pY, L_Y - \epsilon E)$ with respect to the relevant twisting. Finally, applying our main result would then prove the existence of an extremal metric on the total space of the induced Riemann surface fibration over $(Bl_pY, L_Y - \epsilon E)$, which is a projective 3-fold. We emphasise that this is a different construction to producing an extremal metric on the natural fibration over $(Y,L_Y)$ and then blowing-up.

Furthermore,  we expect situations to which our construction applies to be rather abundant, since stability of $p$ (and hence the existence of twisted cscK or extremal metrics) is expected to be a very common property. In fact, through the link with stability, we hope to produce new examples in the following situation. Suppose $p: B\to \M$ is a Fano map, which means that $L = -K_B-L_{CM}$ is ample. The analogue of the Yau-Tian-Donaldson conjecture described in \cite{dervan-ross,twisted} predicts that the existence of a twisted K\"ahler-Einstein metric (hence twisted cscK) on  $p$ is equivalent to the map $p$ being equivariantly K-stable. There are now several ways of producing explicit examples of (equivariantly) K-stable Fano \emph{varieties}, for example the threefolds with a two-dimensional torus action of Ilten-S\"uss \cite{ilten-suess}, which admit K\"ahler-Einstein metrics by the work of Chen-Donaldson-Sun \cite{CDS} and Datar-Sz\'ekelyhidi \cite{datar-szekelyhidi}. Thus by analogy, assuming an analogue of the Yau-Tian-Donaldson conjecture for Fano maps can be proved, we hope to produce cscK metrics on fibrations, where the base is a Fano map. We leave this for future work. 

\vspace{4mm} \noindent {\bf Outline:} We begin in Section \ref{sec:prelims} with preliminaries detailing the geometry of fibrations and giving some basic information on twisted extremal metrics.  Section \ref{sec:lifting} calculates the linearisation of the twisted extremal operator, and gives an interpretation of its kernel as vector fields preserving the map. The moduli theory is also discussed in Section \ref{sec:lifting}, leading to the fact that the vector fields in the appropriate kernel are precisely those that lift to the total space of the fibration. Section \ref{sec:invariance} continues with general theory of maps, proving that the existence of a twisted extremal metric imposes certain constraints on the automorphism group of the map, proving Theorems \ref{intro:unique} and \ref{intro:red}. We prove the existence of an approximately extremal metric in Section \ref{sec:approx}, and use an implicit function theorem argument to conclude the existence of a genuine extremal metric in Section \ref{sec:ift}.

 \vspace{4mm} \noindent {\bf  Acknowledgements:} RD received funding from the ANR grant ``GRACK''. LS received funding from CIRGET.

\vspace{4mm} \noindent {\bf Notation:} For a holomorphic submersion $\pi: X\to B$, the fibres will be denoted $X_b$ for $b\in B$ and are always assumed to have fixed dimension $m$. We suppress pullbacks of line bundles under this map, and use additive notation for line bundles, so that if $H$ is a relatively ample line bundle on $X$ and $B$ is an ample line bundle on $B$ then $rL+H$ denotes the line bundle $(\pi^*L)^{\otimes r} \otimes H$ on $X$.  

We will often have a map $p: B\to M$, and will let $\mfb$ denote the Lie algebra of $\Aut(B,L)$, with $\mfp$ the Lie subalgebra of holomorphic vector fields that also preserve $p$ (in a sense that we shall define). We denote $\mfh$ the Lie algebra of holomorphic vector fields on $X$ which have a zero (hence lift to one, and so any, ample line bundle). We shall see that when $M$ is a moduli space, one can identify $\mfh$ and $\mfb$ naturally. Fixing a maximal compact torus $\mft$ of $\mfp$ will then induce a maximal torus $\mfl$ of $\mfh$. Given a K\"ahler metric, the potentials for these vector fields will be denoted with a bar, for example $\overline \mfh$. The extremal vector field on $B$ will be denoted $\nu$. 

\section{Preliminaries}\label{sec:prelims}

\subsection{Fibrations}\label{subsec:fibrations} We recall the setup which is largely the same as considered by Fine \cite{fine1}. Let $\pi: X\to B$ be a holomorphic submersion. We denote by $\scH =  \pi^*TB$ the horizontal tangent bundle, and denote by $\scV$ the vertical tangent bundle. The fibration structure induces a short exact sequence of holomorphic vector bundles on $X$ $$0 \to \scV\to TX \to \scH \to 0.$$ 

Suppose now that $H$ is a relatively ample line bundle on $X$, endowed with a relatively K\"ahler metric $\omega_0 \in c_1(H)$. The main example is when each fibre $(X_b,H_b)$ admits a cscK metric and has discrete automorphism group, as then these metrics glue to a relatively K\"ahler metric on $H$. This in turn induces a metric on $\scV$ and hence one obtains a metric on $\det \scV^*$, which is simply the fibrewise Ricci curvature of $\omega_0$ \cite[Section 3.1]{fine1}. Denote this form by  $\rho \in c_1(\scV) = c_1(\det \scV)$. Note that the determinant $\det \scV$ is simply $-K_{X/B}$, the relative anti-canonical class.

Since $\omega_0$ restricts to a non-degenerate metric on each fibre, one obtains a splitting in the smooth category $TX \cong \scV \oplus \scH$ by $$\scH_x = \{ u \in T_xX\ |\ \omega_0(u,v) = 0 \textrm{ for all } v \in \scV\}.$$ This further defines splittings in the smooth category of any tensor on $X$ and, following \cite[Section 3.1]{fine1}, to any tensor we shall refer to the components purely from $\scV$ and $\scH$ as the \emph{vertical} and \emph{horizontal} components, respectively. We will also let $\omega_b$ denote the metric on $\scV$ induced by $\omega_0$.

Given a function $\phi$ on $X$ one obtains a function $\int_{X/B}\phi\omega^m$ by integrating over the fibres, i.e. $$\left(\int_{X/B}\phi\omega^m\right)(b) = \int_{X_b} \phi|_{X_b} \omega_b^m,$$ where $m$ is the fibre dimension. One can generalise this as follows. For a $(p,p)$-form $\eta$ on $X$ one obtains a $(p-m,p-m)$ form $\int_{X/B}\eta$ on $B$ by first using the submersion structure to write $\eta = \psi \wedge \pi^*\kappa$ for $\kappa$ a $(p-m,p-m)$-form on $B$,  and then defining $$\int_{X/B} \eta = \left(\int_{X/B}\psi\right)\kappa.$$ As in Fine's work \cite{fine1,fine2}, applying this to the horizontal component of $\rho$ gives a $(1,1)$-form $\alpha$ on $B$, which is closed. Setting $\rho_{H}$ to be the horizontal component of $\rho$, explicitly we have $$\alpha = -\frac{\int_{X/B}(\rho_{H}\wedge \omega_0^m)}{\int_{X/B}\omega_0^m}.$$

From this one obtains a splitting $$C^{\infty}(X) \cong C^{\infty}_0(X) \oplus C^{\infty}(B),$$ where $C^{\infty}_0(X)$ denotes the functions which integrate to zero over the fibres (i.e. $\phi \in C^{\infty}_0(X)$ is equivalent to $\int_{X/B} \phi \omega^m = 0$).

\subsection{Twisted extremal metrics}\label{subsec:extremal}  Let $B$ be a K\"ahler manifold endowed with an ample line bundle $L$. We say that a K\"ahler metric $\omega\in c_1(L)$ is a \emph{constant scalar curvature K\"ahler (cscK) metric} if $$S(\omega) = \Lambda_{\omega} \Ric\omega = c,$$ where $$c=n\frac{-K_X.L^{n-1}}{L^n} = n\frac {\int_X \Ric\omega\wedge\omega^{n-1}}{\int_X \omega^n}$$ is the only possible topological constant and $n=\dim B$. We say further that $\omega$ is \emph{extremal} if $\bar \partial \nabla^{1,0}S(\omega)=0$. 

\begin{remark}\label{fibre-independence} If $\pi: X \to B$ is a holomorphic submersion and $H$ is a line bundle on $X$ with restriction $H_b = H|_{X_b}$ to a fibre $X_b$, then  the intersection number $$\frac{-K_{X_b}.H_b^{n-1}}{H_b^n}$$ is actually independent of $b$. There are various ways of seeing this, one of which is to note that $H_b^n$ and $-K_{X_b}.H_b^{n-1}$ are coefficients of the Hilbert polynomial $h(r) = \chi(X_b,H_b^{\otimes r})$ and to use that the Hilbert polynomial is constant in flat families. More directly, one could similarly note that intersection numbers are also constant in flat families. Another more differential-geometric proof is to note that, if $\omega \in c_1(H)$ is relatively K\"ahler, then $\int_{X_b} \omega_b^{n}$ is continuous as one varies $b \in B$, but since it is also an integer it must be independent of $b$. In particular, if $\omega_b$ is a cscK metric on $X_b$ for all $b \in B$, then the scalar curvature of $\omega_b$ is independent of $b$. \end{remark}

Suppose in addition that $\alpha$ is a semipositive $(1,1)$ form on $B$. The typical case we shall consider is when $\alpha$ is the pullback of a K\"ahler metric from a map $B\to M$. We say that $\omega \in c_1(H)$ is a \emph{twisted cscK metric} if $$S(\omega) - \Lambda_{\omega} \alpha = c,$$ where $c$ is the appropriate topological constant. In this form these metrics were first studied by Fine \cite{fine1}, though the study of the special case of twisted K\"ahler-Einstein metrics goes back at least to Yau. More generally we say that $\omega$ is \emph{twisted extremal} if $\bar \partial \nabla^{1,0}(S(\omega) - \Lambda_{\omega} \alpha) = 0$. Clearly a twisted cscK metric is extremal, and conversely we shall see that a twisted extremal metric is twisted cscK if $B$ admits no holomorphic vector fields.  

For a K\"ahler manifold manifold $(X,A)$ with an ample line bundle $A$, we define $\Aut(X,A)$ to be the automorphisms of $X$ which lift to $A$. Its Lie algebra $\mfh$ consists of holomorphic vector fields whose flows lift to $L$. By a well known result of LeBrun-Simanca \cite[Theorem 1]{lebrun-simanca}, these vector fields are precisely those which vanish somewhere, and in particular are independent the choice of ample line bundle $A$. Fixing a metric $\omega \in c_1(A)$ with induced Riemannian metric $g$, each element $\xi \in\mfh$ can be written in coordinates, and using the obvious notation, as $$\xi^j = g^{j\bar k} \partial_{\bar k} f,$$ for some $f\in C^{\infty}( X, \C)$, unique up to the addition of a constant, called a \emph{holomorphy potential} for $\xi$. We denote by $\overline\mfh$ the holomorphy potentials of the holomorphic vector fields in $\mfh$.

The isomorphism $T^{1,0}X \cong TX$ between the holomorphic and real tangent bundles gives a correspondence $$g^{j\bar k} \partial_{\bar k} f \to \frac{1}{2}(\nabla u + J \nabla v),$$ where $f = u+iv$ is the decomposition of $f$ into real and imaginary parts and $J$ is the complex structure. We also denote by $\mft\subset \mfh$ the vector fields in $\mfh$ which correspond to Killing vector fields under this isomorphism, which by above are vector fields with purely imaginary holomorphy potential.

Denote by \begin{align*}\mathcal{D}_{\omega}: C^{\infty}(X,\C) &\to \Gamma(T^{1,0}X \otimes  \Omega^{0,1}(X)), \\  \mathcal{D}_{\omega} &= \bar\partial \nabla^{1,0},\end{align*} and recall the \emph{Lichnerowicz operator} is defined as $\mathcal{D}^*_{\omega} \mathcal{D}_{\omega} $. We shall use two key facts regarding this operator: firstly, it is the linearisation of the scalar curvature if $\omega$ is a cscK metric; secondly, the kernel $\ker \mathcal{D}^*_{\omega} \mathcal{D}_{\omega} = \ker\mathcal{D}_{\omega}$ consists precisely of holomorphy potentials $f$, see \cite{calabi-2}. Denoting $\ker_0 \mathcal{D}^*_{\omega} \mathcal{D}_{\omega}$ the kernel restricted to functions which integrate to zero with respect to the volume form induced by $\omega$, it follows that this operator induces an isomorphism  \begin{align*}\ker_0 \mathcal{D}^*_{\omega} \mathcal{D}_{\omega} &\to \mfh, \\ f  &\to g^{j\bar k} \partial_{\bar k} f.\end{align*} Using this operator one sees that a K\"ahler metric is extremal precisely if its scalar curvature is a holomorphy potential, or equivalently $\mathcal{D}_{\omega}S(\omega) = 0$. Similarly a K\"ahler metric on $B$ is a twisted extremal metric if and only if $S(\omega) - \Lambda_{\omega}\alpha$ is a holomorphy potential, i.e. it lies in $\overline\mfb$ and so corresponds to an element of the Lie algebra  $\mfb$ of $\Aut(B,L)$.

The twisted extremal equation can therefore be written $$ S (\omega)  - \Lambda_{\omega} \alpha - f = 0 ,$$ where $f$ is the potential of some holomorphic vector field $\nu\in\mfb$ with respect to the K\"ahler metric $\omega$. If we consider a perturbation $\omega + i \partial \overline{\partial} \phi$ of $\omega$ within its K\"ahler class, then the potential for $\nu$ changes by $\nu ( \phi) =  \frac{1}{2} \langle \nabla f, \nabla \phi \rangle$, where the inner product and gradient is taken using $\omega$ (see for example \cite[Lemma 12]{szekelyhidi-extremal-blowups} or \cite{APS}). Writing the twisted extremal equation as an equation for $\phi$ gives $$ S (\omega + i \partial \overline{\partial} \phi ) -  \Lambda_{\omega+ i \partial \overline{\partial} \phi} \alpha -  \frac{1}{2} \langle \nabla f, \nabla \phi \rangle  - f  = 0 .$$ Thus to solve the \emph{extremal} equation, we must find a  zero of the map
\begin{align*}&P: C^{4,\alpha} (B) \times \overline{\mathfrak{b}} \rightarrow C^{0, \alpha} (B), \\  &P(\phi, f)= S (\omega + i \partial \overline{\partial} \phi ) - \frac{1}{2} \langle \nabla f, \nabla \phi \rangle  - f.
\end{align*} A standard elliptic regularity argument then implies a $C^{4,\alpha}$-extremal metric is actually smooth.

\section{Lifting vector fields}\label{sec:lifting}

\subsection{Linearising the twisted extremal operator}

One of the key links between cscK metrics and the geometry of K\"ahler manifolds is through the Lichnerowicz operator, whose kernel is precisely given by the holomorphy potentials of the K\"ahler manifold. The goal of this section is to prove an analogous result for twisted cscK and extremal metrics. We begin with the following calculation of the linearisation of the twisted extremal operator: 

\begin{proposition}\label{prop:linearisation-twisted} The linearisation of the twisted extremal operator $P$ is given as \begin{align}\label{twistedextremaloperator} dP_{(0,f)} ( \phi, h) = - \mathcal{D}^*_{\omega} \mathcal{D}_{\omega} (\phi) + \frac{1}{2} \langle \nabla \big( S(\omega) - f \big), \nabla \phi \rangle + \langle i \partial \overline{\partial} \phi, \alpha \rangle_{\omega} -h.\end{align} \end{proposition}

\begin{proof}

Recall that the linearisation of the extremal operator $\mathcal{Q}: C^{4,\alpha} (B) \times \overline{\mathfrak{h}}\rightarrow C^{0, \alpha} (B)$ \begin{equation}\label{extremaloperator}\mathcal{Q} (\phi, f) = S (\omega + i \partial \overline{\partial} \phi ) -  \frac{1}{2} \langle \nabla f, \nabla \phi \rangle  - f \end{equation} is given by $$ d\mathcal{Q}_{(0,f)} ( \phi, h) = - \mathcal{D}^*_{\omega} \mathcal{D}_{\omega} (\phi) + \frac{1}{2} \langle \nabla \big( S(\omega) - f \big), \nabla \phi \rangle - h.$$

Thus all that remains is to subtract the linearisation of the operator $\phi \to \Lambda_{\omega_{\phi}} \alpha$, where $\omega_{\phi} = \omega + \ddb \phi$. Writing $$(\Lambda_{\omega_{\phi}} \alpha) \omega_{\phi}^n = n \alpha \wedge \omega_{\phi}^{n-1}$$and expanding one sees that the linearisation is given by $$\phi \to \frac{ n (n-1) i \partial \overline{\partial} \phi \wedge \alpha \wedge \omega^{n-2}  }{\omega^n} - \Lambda_{\omega} \alpha \cdot \Delta(\phi) .$$ We claim that $$ - \langle i \partial \overline{\partial} \phi, \alpha \rangle_{\omega} = \frac{ n (n-1) i \partial \overline{\partial} \phi \wedge \alpha \wedge \omega^{n-2}  }{\omega^n} - \Lambda_{\omega} \alpha \cdot \Delta(\phi) .$$ Using the formula $$ n (n-1) \beta_1 \wedge \beta_2 \wedge \omega^{n-2} = \big( \Lambda_{\omega} \beta_1 \Lambda_{\omega} \beta_2 - \langle \beta_1, \beta_2 \rangle_{\omega} \big) \omega^n $$ for $(1,1)$-forms $\beta_1, \beta_2$ \cite[Lemma 4.7]{szekelyhidi-book}, together with $\Delta (\phi) = \Lambda_{\omega} (i \partial \overline{\partial} \phi)$, gives
\begin{equation}\label{doublewedge} - \langle i \partial \overline{\partial} \phi, \alpha \rangle_{\omega}\omega^n = n (n-1) i \partial \overline{\partial} \phi \wedge \alpha \wedge \omega^{n-2}  
 - \Lambda_{\omega} \alpha \cdot \Delta (\phi)\omega^n,\end{equation}
as required.
\end{proof}

While the above results holds without any assumptions on $\alpha$, it will now be useful to assume $\alpha$ is semipositive. Recall from Section \ref{subsec:extremal} that $\D^*\D \phi = 0$ if and only if $\phi$ is a holomorphy potential, and so the kernel of the linearisation of the extremal operator \eqref{extremaloperator} at an extremal metric consists precisely of holomorphy potentials. The analogue of this statement for the twisted extremal operator is the following, which is the operator obtained by choosing $f=S(\omega) - \Lambda_{\omega} \alpha$ and restricting to $h=0$ in the operator \eqref{twistedextremaloperator}.

\begin{proposition}Define an operator $$\scL_{\alpha}(\phi)= - \mathcal{D}^*_{\omega} \mathcal{D}_{\omega} (\phi) + \frac{1}{2} \langle \nabla \Lambda_{\omega} \alpha,  \nabla \phi \rangle  +  \langle i \partial \overline{\partial} \phi, \alpha \rangle.$$ Suppose $\alpha$ is semipositive. Then the kernel of $\scL_{\alpha}$ consists of holomorphy potentials satisfying $| \nabla \phi |_{\alpha}  =0$. \end{proposition}

\begin{proof}  First note that $\scL_{\alpha}(\phi)=0$ if and only if for all $\psi \in C^{\infty}(X)$  we have
$$0 =\int_B \scL_{\alpha} (\phi)\psi\omega^n = \int_B \big( - \mathcal{D}^*_{\omega} \mathcal{D}_{\omega} (\phi)  + \frac{1}{2} \langle \nabla \Lambda_{\omega} \alpha,  \nabla \phi \rangle + \langle i \partial \overline{\partial} \phi, \alpha \rangle \big)\psi  \omega^n.$$ The Lichnerowicz operator $\mathcal{D}^*_{\omega} \mathcal{D}_{\omega} $ is self-adjoint and hence satisfies  $$\int_B \mathcal{D}^*_{\omega} \mathcal{D}_{\omega} (\phi)\psi  \omega^n = \int_B  \langle \mathcal{D}_{\omega} (\phi), \mathcal{D}_{\omega} (\psi) \rangle  \omega^n.$$ Using equation \eqref{doublewedge}, the remaining part of the integral equals 
\begin{align}\label{keystep}\int_B (\scL_{\alpha} (\phi) + \mathcal{D}^*_{\omega} \mathcal{D}_{\omega} (\phi))\psi\omega^n = \int_B \frac{1}{2}& \langle \nabla \Lambda_{\omega} \alpha, \psi \nabla \phi \rangle \omega^n + \int_B \psi \Delta (\phi) \Lambda_{\omega} \alpha \omega^n-\\&- n (n-1)\int_B  \psi i \partial \overline{\partial} \phi \wedge  \alpha  \wedge \omega^{n-2}.\nonumber \end{align}

Applying the Leibniz rule for $\nabla^*$, we see that
\begin{align*} \int_B  \langle \nabla \Lambda_{\omega} \alpha, \psi \nabla \phi \rangle \omega^n &= \int_B    \Lambda_{\omega} \alpha \cdot \nabla^* ( \psi \nabla \phi )  \omega^n\\
&= \int_B \Lambda_{\omega} \alpha \cdot ( \psi \nabla^* \nabla \phi - \nabla \psi \cdot \nabla \phi ) \omega^n \\
&= - 2 \int_B  \psi  \Delta( \phi) \Lambda_{\omega} \alpha \omega^n - \int_B \nabla \psi \cdot \nabla \phi \Lambda_{\omega} \alpha \omega^n
\end{align*}
since $\nabla^* \nabla = - 2 \Delta$ on functions.  Thus in equation \eqref{keystep}, the two first terms equal $-\frac{1}{2}\int_B \nabla \psi \cdot \nabla \phi \Lambda_{\omega} \alpha \omega^n.$ 

For the remaining term, we begin with $$ n (n-1)\int_B \psi i \partial \overline{\partial} \phi \wedge  \alpha  \wedge \omega^{n-2} = -n (n-1) \int_B i \partial \psi \wedge \overline{\partial} \phi \wedge \alpha \wedge \omega^{n-2}.$$  As above using \cite[Lemma 4.7]{szekelyhidi-book} gives $$n(n-1)  i\partial \psi \wedge \overline{\partial} \phi \wedge \alpha \wedge \omega^{n-2} = \left( \Lambda_{\omega} (i\partial \psi \wedge \overline{\partial} \phi) \cdot \Lambda_{\omega} \alpha - \langle i \partial \psi \wedge \overline{\partial} \phi , \alpha \rangle_{\omega} \right) \omega^n .$$ For the first term, note that $$ \Lambda_{\omega} ( i\partial \psi \wedge \overline{\partial} \phi ) = g^{p \overline{q}} \partial_p \psi \partial_{\overline{q}} \phi = \frac{1}{2}\langle \nabla \psi,\nabla \phi\rangle_{\omega}.$$ Similarly for the second term $$ \langle i\partial \psi \wedge \overline{\partial} \phi , \alpha \rangle_{\omega} = g^{p \overline{q}} g^{r \overline{s}} \partial_p \psi \partial_{\overline{s}} \phi \alpha_{r \overline{q}} = (\nabla \phi)^{r}(\nabla \psi)^{\overline{q}} \alpha_{r \overline{q}} = \langle \nabla \psi, \nabla \phi \rangle_{\alpha}.$$ So $$n(n-1)  i\partial \psi \wedge \overline{\partial} \phi \wedge \alpha \wedge \omega^{n-2} = \left( \frac{1}{2}\langle \nabla \psi,\nabla \phi\rangle_{\omega}\cdot \Lambda_{\omega} \alpha -\langle \nabla \psi, \nabla \phi \rangle_{\alpha} \right) \omega^n .$$ It follows that $$- n (n-1)\int_B  \psi i \partial \overline{\partial} \phi \wedge  \alpha  \wedge \omega^{n-2} = \int_B \left( \frac{1}{2}\langle \nabla \psi,\nabla \phi\rangle_{\omega}\cdot \Lambda_{\omega} \alpha -\langle \nabla \psi, \nabla \phi \rangle_{\alpha} \right) \omega^n,$$ and summing up gives \begin{equation}\label{eq:needs-conj}\int_B \scL_{\alpha} (\phi)\psi\omega^n= - \int_B  \langle \mathcal{D}_{\omega} (\phi) , \mathcal{D}_{\omega} (\psi)\rangle \omega^n - \int_B \langle \nabla \psi, \nabla \phi \rangle_{\alpha}\omega^n. \end{equation}

Setting $\psi=\phi$, we see that $\int_B \psi\scL_{\alpha} (\phi)\omega^n=0$ implies $\int_B  | \mathcal{D}_{\omega} (\phi)|^2 \omega^n=0$ and $\int_B |\nabla \phi|^2_{\alpha}\omega^n=0$. These respectively imply $\mathcal{D}_{\omega} (\phi)=0$ and $|\nabla \phi|_{\alpha}=0$; remark that the first is equivalent to $\phi$ being a holomorphy potential.

Conversely suppose $\mathcal{D}_{\omega} (\phi)=0$ and $|\nabla \phi|_{\alpha}=0$. Since $\alpha$ is semipositive, one has $| \nabla \phi |_{\alpha}=0$ if and only if $\langle \nabla \phi,\nabla \psi\rangle_{\alpha} =0$ for all $\psi \in C^{\infty}(X)$. But this clearly implies $\int_B \psi\scL_{\alpha} (\phi)\omega^n=0$ for all $\psi$, proving the result.
\end{proof}

\begin{remark}
It seems unlikely that there is any reasonable interpretation for the kernel when $\alpha$ is not semipositive.
\end{remark}

\subsection{Automorphisms of maps}

As described in the introduction, the most natural situation in which to study twisted extremal metrics is when one has a map of K\"ahler manifolds $p: (B,\omega) \to (M,\alpha)$. Specialising to this case, we relate the relevant kernel to automorphisms of $p$.

\begin{definition} We say that $g \in \Aut(B,L)$ is an automorphism of $p: B\to M$ if $p(g(b)) = p(b)$ for all $b \in B$. We write the subgroup of automorphisms of $p$ as $\Aut(p)$, and denote its Lie algebra  by $\mfp$. 
 \end{definition}
 
\begin{proposition} \label{automs-prop}
Suppose $\rho(t)\in \Aut(B)$ is the flow of a holomorphic vector field $\mu \in  \mfb$. Then $\rho(t) \in \Aut(p)$ if and only if the holomorphy potential $\phi$ of $v$ is in the kernel of $\scL_{\alpha}$. The same result holds if $\alpha$ is a smooth $(1,1)$-form which is K\"ahler on a Zariski open locus of $M$.
\end{proposition}

\begin{proof} A holomorphy potential is, in particular, a smooth function. Thus a holomorphy potential $\phi$ is in the kernel of $\scL_{\alpha}$ if and only if $\scL_{\alpha}(\phi) = 0$ holds \emph{on a Zariski open locus of $B$}. Similarly for a given $g\in \Aut(B)$, by definition $g \in \Aut(p)$ if $p(g(b)) = p(b)$ for all $b \in B$, which is true if and only if $p(g(b)) = p(b)$ for all $b$ in a Zariski open subset of $B$, since $g$ is a holomorphic map, and Zariski open sets are dense in the analytic topology.

Thus to prove the result we can work on a Zariski open subsets of $B$ and similarly $M$. By replacing $M$ with its image we assume $p$ is surjective. The Stein factorisation of the map $p: B\to M$ is $ B\overset{p_1}{\to}M' \overset{p_2}{\to} M$ where $p_1$ is a contraction and $p_2$ is finite. Note that the automorphisms of $p$ are equivalent to the isomorphisms of $p_1$. As $p_2$ is finite, it is unramified in codimension one, hence the pullback $p_2^*\alpha$ is positive on a Zariski open locus of $M'$. Thus we may assume $p$ is actually a contraction.

We next work on the Zariski open locus of $B$ on which $p: B \to M$ is a holomorphic submersion, which exists by \cite[p106]{GPR}. In fact working first on the Zariski open locus of $M$ which is smooth, it is then clear $p$ is a submersion on a Zariski open locus as the condition is simply that the Jacobian of the derivative of the map has its largest rank. Let $B^o \subset B$ be such a Zariski open submanifold.

As $p: B^o \to M$ is a submersion, locally around any point there are holomorphic coordinates such that $p$ is a projection onto some of the subset of the coordinates. In a neighbourhood of $b \in B^o$, pick coordinates $x_1,\hdots,x_i,x_{i+1},\hdots,x_n$ such that $p(x_1,\hdots,x_i,x_{i+1},\hdots,x_n) = (x_{i+1},\hdots,x_n)$. Since $\alpha$ is pulled back from $M$, $\alpha$ is strictly positive in the coordinates $(x_{i+1},\hdots,x_n)$ and zero in the other directions. Thus $| \nabla \phi |_{\alpha}  =0$ if and only if the flow of $\phi$ acts only on $(x_1,\hdots,x_i)$. But this happens for all $b \in B^o$ if and only if $\rho(t) \in \Aut(p)$, as required. \end{proof}

\begin{remark}\label{general-lift} It should come as no surprise that one only needs $\alpha$ to be positive on a Zariski open locus of $M$ in the above result. The typical situation in which this occurs is when $q: M \to M'$ is a birational morphism, so that $q^*\alpha'$ is positive away from the exceptional locus of $q$. But in this situation, the automorphisms of a map $p: B\to M$ are equal to the automorphisms of $p \circ q$, justifying the expectation. 

 Slightly more generally, Proposition \ref{automs-prop} holds in the situation that one has a map $p: M^o \to B^o$ where $M^o \subset M$ and $B^o \subset B$ are Zariski open loci of $M$ and $B$ respectively, with $\alpha$ a K\"ahler metric on $B^o$ which extends to a continuous metric on $M$.\end{remark}

\begin{remark}\label{rmk:uniqueness} Using these ideas, one can give a geometric interpretation of the uniqueness of twisted cscK and extremal metrics. It is proven by Berman-Berndtsson that an extremal metric is unique up to isometries \cite[Section 4]{berman-berndtsson}. In Section \ref{sec:invariance} we shall see that a twisted extremal metric automatically has isometry group which is a maximal compact subgroup of $\Aut(p)$. Thus it is natural to ask if the same uniqueness holds for twisted extremal metrics. Absolute uniqueness is proved by Berman-Berndtsson in the case that $\alpha$ is positive, which is the case that $p$ is an embedding.

In fact, the techniques of Berman-Berndtsson apply more generally. We consider first the (twisted) cscK case. The key tool used in \cite{berman-berndtsson} is the Mabuchi functional, which is a functional on the space of K\"ahler metrics $\omega + \ddb\phi $, and has a natural twisted analogue as follows. Consider a path $\omega_t=\omega+\ddb \phi_t$ with $\omega_0 = \omega$. The twisted Mabuchi functional is then $$\M_{\alpha}(\phi) = -\int_0^1\int_B \dot \phi_t (S(\omega_t) - \Lambda_{\omega_t}\alpha - C)\omega_t^n.$$ Berman-Berndtsson show that, in a suitable sense, this functional is convex along geodesics in the space of weak K\"ahler metrics with $L^{\infty}$-coefficients (this is proven when $\alpha$ is strictly positive, but the same result holds by approximating $\alpha$ by positive forms, or by an examination of the proof in the positive case). One calculates, using the same ideas as Proposition \ref{prop:linearisation-twisted}, that the Hessian $H(\phi,\psi)$ of the twisted Mabuchi functional is given as: $$H(\phi,\psi) =  \int_B  \mathcal{D}_{\omega} (\psi) \mathcal{D}_{\omega} (\phi) \omega^n + \int_B \langle \nabla \psi, \nabla \phi \rangle_{\alpha}\omega^n.$$ Thus from Proposition \ref{automs-prop} one sees that the Hessian of the Mabuchi functional is non-negative, and degenerates precisely along the holomorphy potentials $\mfp$ which preserve the map $p$. This infinitesimal non-degeneracy, together with convexity along geodesics, is exactly what Berman-Berndtsson use to obtain uniqueness of cscK metrics up to automorphisms. Then the same argument applies to give uniqueness of twisted cscK metrics in the following sense: if $\omega,\omega'$ are two twisted cscK metrics, then there is an element $g \in \Aut(p)$ such that $\omega =g^*\omega'$. 

The argument for uniqueness of twisted extremal metrics is similar, the main difference is one works only on K\"ahler metrics invariant under the imaginary part of the extremal vector field, and modifies the Mabuchi functional by adding an extra term. We refer to \cite[Section 4]{berman-berndtsson} for further details in the untwisted case, which apply in a straightforward manner in the twisted setting.  

Using this allows one to give a geometric interpretation to Fine's uniqueness assumption \cite{fine1}. Indeed, we shall see in Remark \ref{rmk:second-uniqueness} that $\omega_B$ is the \emph{unique} twisted cscK metric if and only if $\Aut(X,kL+H)$ is discrete. 
\end{remark}

\begin{corollary} If $\omega, \omega' \in c_1(L)$ are both twisted cscK, then there is a $g \in \Aut(p)$ with $g^*\omega = \omega'$. If $\omega, \omega'$ are twisted extremal metrics with the same extremal vector field $v$, then denoting by $\Aut_v(p)$ the automorphisms of $p$ which commute with the flow of $v$, there is a $g \in \Aut_v(p)$ with $g^*\omega = \omega'$. \end{corollary}

Since, as we will later show in  Corollary \ref{thm:isometries} , twisted extremal metrics are invariant under a maximal compact subgroup of $\Aut_v(p)$, it follows that one could take a $g\in  \Aut_v(p)$ to be in the connected component of the identity.

\subsection{The Weil-Petersson metric}

We consider to a fibration $X\to B$ such that $H$ is a relatively ample line bundle with $\omega_0 \in c_1(H)$, such that each fibre $(X_b,H_b)$ is a smooth projective manifold with discrete automorphism group and the restriction $\omega_b$ of  $\omega_0$ to each fibre is cscK with scalar curvature $S(\omega_b)$ (which is independent of $b$). For the moment we allow $B$ to be non-compact. As in Section \ref{subsec:fibrations}, the metric $\omega_0$ form $\rho \in c_1(K_{X/B})$ in the relative anticanonical class. 

\begin{definition}\cite[Theorem 7.8]{fujiki-schumacher} Denote by $[\omega_b]^m = \int_{X_b} \omega_b^m$ the volume of the fibres. We define the \emph{Weil-Petersson metric} of $X\to B$ to be the $(1,1)$-form on $B$ given as $$\omega_{WP} = \frac{S(\omega_b)}{(m+1)}\frac{ \int_{X/ B} \omega^{m+1} }{\int_{X/B}\omega^m}-\frac{ \int_{X/B} \rho\wedge \omega^m}{\int_{X/B}\omega^m},$$ where $m$ is the dimension of the fibres, and the fibre integrals are simply the pushforwards of Section \ref{sec:prelims}. Note that $\int_{X/B}\omega^m$ is independent of $b\in B$ and is simply the volume of any fibre.
\end{definition}

\begin{remark}\label{submersion-of-spaces} Similarly one can define the Weil-Petersson metric over a submersion $X\to B$ where $X$ and $B$ are reduced complex spaces; we refer to \cite{fujiki-schumacher} for further details. \end{remark}

Remark that $\omega_{WP}$ is closed since the differential $d$ commutes with pushforward, see e.g. \cite{fujiki-schumacher}. Since the pushforward is of an $(m+1,m+1)$-form and each fibre is $m$-dimensional, the resulting form on $B$ is a $(1,1)$-form. The above is not the main representation of the Weil-Petersson metric used by Fujiki-Schumacher, but is equivalent by their results \cite{fujiki-schumacher}.

\begin{lemma}\label{WP-alpha} The form $\alpha$ is the Weil-Petersson metric induced from the fibration $X\to B$. \end{lemma}

\begin{proof} This is essentially proved in \cite[Lemma 2.3]{fine2}; we recall Fine's argument for the reader's convenience. 

Recall that $\alpha$ is defined by first taking the horizontal component $\rho_H$ of $\rho$ and defining $$\alpha = -\frac{\int_{X/B}(\rho_{H}\wedge \omega^m)}{\int_{X/B}\omega^m}.$$ On each fibre, the vertical component $\rho_V$ of $\rho$ is the Ricci curvature of $\omega_b$  \cite[Lemma 3.3]{fine1}. Since $\omega_b$ is cscK we therefore have $$m\rho_V \wedge \omega_b^{m-1} = S(\omega_b) \omega_b^m.$$ From here the result is a simple computation. Indeed \begin{align*}\int_{X/B} \rho_V\wedge \omega^m &= \int_{X/B} m\rho_V \wedge \omega_b^{m-1}\wedge \omega_H, \\
&= \int_{X/B} \omega_b^m \wedge \omega_H, \\ &= \frac{S(\omega_b)}{m+1}\int_{X/B} \omega^{m+1}. \end{align*} But then \begin{align*}\int_{X/B} (\rho \wedge \omega^m) &= \int_{X/B} (\rho_V + \rho_H)\wedge \omega^m, \\ &= -\left(\int_{X/B}\omega^m\right)\alpha +\frac{S(\omega_b)}{m+1}\int_{X/B} \omega^{m+1},\end{align*}giving the result. \end{proof}

\begin{remark}When the fibres $X_b$ are curves, Lemma \ref{WP-alpha} was noticed by Fine \cite[Theorem 3.5]{fine1}. In general Fine remarks in \cite{fine2} that $[\alpha] = [\omega_{WP}]$, what we wish to point out here is that even the forms themselves are equal. \end{remark}

The first result of Fujiki-Schumacher regarding the Weil-Petersson metric we shall use is the following. 

\begin{theorem}\cite[Theorem 7.4]{fujiki-schumacher} The Weil-Petersson metric is semi-positive. If the fibres of $X\to B$ are pairwise biholomorphically distinct, then $\omega_{WP}$ is K\"ahler. \end{theorem}

\begin{remark} Fujiki-Schumacher prove the above under the somewhat more general hypothesis that the family is effective, which means that the Kuranishi map is injective (i.e. the fibres locally have distinct holomorphic structures). \end{remark}

The second result of Fujiki-Schumacher we shall use concerns the existence of a moduli space of polarised manifolds admitting cscK metrics. We remark that while it is expected that such a manifold is a quasi-projective variety, this is far from known.

\begin{theorem}\cite[Theorem 6.6]{fujiki-schumacher} There exists a reduced complex space $\M$ which is a moduli space of polarised manifolds which admit cscK metrics and have discrete automorphism group. The Weil-Petersson metric is a K\"ahler metric on this space. \end{theorem}

The result above should be understood as follows. A given point $x \in \M$ represents a polarised manifold which may have discrete, non-trivial automorphism group. Using this, Fujiki-Schumacher give $\M$ the structure of an analytic Deligne-Mumford stack (also called ``orbi-spaces'' or ``V-spaces''). That is, around each $x\in \M$ there is an open set $U \subset \M$, a group $G$ and a quotient $V/G \cong U$, where $V$ is a (genuine) analytic space. In our case $G$ is the automorphism group of the polarised manifold which $x$ represents. Morphisms of such spaces of are defined to be ones which lift to the finite covers locally. The Weil-Petersson metric is then a K\"ahler orbi-metric on $\M$ (i.e. a $G$-invariant metric on the covers), which is possible since a cscK metric is invariant under the (finite) automorphism group of the polarised manifold. We refer to \cite[Definition 1.5]{fujiki-schumacher} for further details on this construction, including a discussion of K\"ahler geometry on analytic spaces (the main point being that a K\"ahler metric is one which, under a local embedding in some $\C^N$, locally extends to a K\"ahler metric on $\C^N$).

Since our techniques allow one to restrict to Zariski open subsets, one can typically ignore the orbi-structure on $\M$ by restricting to a subset. In any case we shall give another construction of the  moduli space relevant to our work using algebro-geometric techniques, which in particular make the universal properties of such a moduli space that we shall require more transparent. 

\begin{remark} From moduli considerations, we expect in Proposition \ref{automs-prop} that the situation that $\alpha$ is merely positive on a Zariski open locus is rather fundamental in the study of twisted cscK metrics. For example, a typical situation in algebraic geometry is a fibration $X \to B$ where the fibres $X_b$ are only smooth on a Zariski open locus of $B$, and the other fibres can be very singular. In this case there is no hope to extend a Weil-Petersson metric as a K\"ahler metric over a relevant moduli space (as the map to the moduli space exists only on a Zariski open locus of $B$), and the Weil-Petersson metric may only extend to a current on $X$. Nevertheless for our proof above we need the Weil-Petersson metric to at least extend to a continuous metric over $B$, which happens for example when $X\to B$ is a fibration where the general fibre is a smooth curve but where a special fibre may be nodal. \end{remark}

The existence of the moduli space, together with Lemma \ref{WP-alpha}, allow us to view the fibration $X\to B$ as induced from a map $p: B \to \M$, with $\M$ endowed with the Weil-Petersson metric.

\subsection{Chow stability and lifts}

The goal of this section is to give another construction of a moduli space which parametrises the fibres of the fibration $X \to B$. This construction uses Geometric Invariant Theory (GIT), and from this construction we will obtain some useful properties from the general theory of GIT \cite{GIT}. More precisely, we will construct the moduli space as a quotient of a Chow variety, which parametrises certain subvarieties in projective space.

We begin with a brief discussion of the Chow variety, referring to \cite[Section 1.3]{kollar} for more details. Fixing some $Y\subset\pr^n=\pr(V)$, we only show how to associate a point in another projective space. When this is done in families, this yields a moduli space, namely the Chow variety, which can be compactified naturally by adding certain cycles at the boundary. We denote by $m$ the dimension of $Y$ and $d$ the degree of $Y$, which means $(\scO_{\pr^n}(1)|_Y)^m=d$. The set $Z$ of $(n-m-1)$-dimensional planes intersecting $Y$ nontrivially is by definition a subvariety of the Grassmanian $\Grass(n-m,n+1)$. One shows that $Z$ is a codimension one subvariety of this Grassmanian, and hence  $Z=\V(f)$ is the vanishing locus of some section $f\in H^0(\Grass(n-m,n+1), \scO(d))$ unique up to scaling. Thus associated to $Y$, we have a section $f\in H^0(\Grass(n-m,n+1), \scO(d))$ unique up to scaling, and so a corresponding point $[f]\in \pr(H^0(\Grass(n-m,n+1), \scO(d)))$, called the Chow point. We denote by $\Ch(\pr^m) \subset \pr(H^0(\Grass(n-m,n+1), \scO(d)))$ the Chow variety induced from this construction, and note that it admits an ample line bundle induced from the line bundle $\scO_{\Ch(1)}$ on $\pr(H^0(\Grass(n-m,n+1), \scO(d)))$.

For any $Y\subset \pr^m$, for any $g\in G:=\SL(m)$ the variety $g \cdot Y \subset \pr^m$ is clearly isomorphic to $Y$. Thus to obtain a moduli space which parametrises subvarieties of $\pr^m$, we wish to take a quotient of $\Ch(\pr^m)$ by the group $\SL(m+1)$, which acts on $\pr(H^0(\Grass(n-m,n+1))$ in a natural way. In order to do this, we use the machinery of GIT, for which we refer to \cite{GIT} for an introduction. This general theory then produces a moduli space $(\M,\scL) = (\Ch(\pr^m),\scO_{\Ch}(1))\git \SL(m+1)$, which is a projective variety. In fact, this quotient is constructed by taking $M = \Proj \oplus_{r \geq 1}H^0(\Ch(\pr^m), \scO_{\Ch}(r))^G$, where for a vector space $W$, we denote by $W^G$ the $G$-invariant sections. Clearly the induced rational map $\Ch(\pr^m) \dashrightarrow M$ is only defined on points $x \in \Ch(\pr^m)$ for which there exists an $r$ and a $G$-invariant section of  $H^0(\Ch(\pr^m), \scO_{\Ch}(r))$ which does not vanish at $x$; let us call $x$ \emph{Chow stable} if this holds and further the stabiliser of $x$ is finite. The fundamental result of GIT is that $\M$ is a \emph{coarse} moduli space when restricted to the stable locus of $\Ch(\pr^m)$, which is Zariski open (of course this holds for any projective variety, not just the Chow variety) \cite{GIT}. We denote by $M^{st}$ the quasi-projective scheme which parametrises Chow stable varieties in $\pr^m$. 

\begin{remark}  $M^{st}$ being a coarse moduli space means that a map $B\to M$ is associated to each family $X \to B$ where each fibre $X_b$ corresponds to a Chow stable variety in $\pr^m$.  Remark that in general no universal family $U\to \M^{st}$ can exist owing to the presence of points $x \in M^{st}$ which have finite but non-trivial stabiliser. If one restricts to the locus of $M^{st}$ for which the stabiliser is trivial, then such a universal family exists (which means that on this locus, $M^{st}$ is a \emph{fine} moduli space). \end{remark}

\begin{definition} We say that a polarised variety $(Y,H)$ is \emph{asymptotically Chow stable} if $Y$ is Chow stable under the Kodaira embeddings $Y \hookrightarrow \pr(H^0(Y,H^r))$ for $r \gg 0$. \end{definition}

The key result we need linking the moduli theory of cscK manifolds and polarised varieties is the following, due to Donaldson \cite{donaldson-projective-embeddings}.

\begin{theorem}\cite{donaldson-projective-embeddings} Suppose $(Y,H)$ is a smooth polarised variety which has discrete automorphism group. If $(Y,H)$ in addition admits a cscK metric, then $(Y,H)$ is asymptotically Chow stable. \end{theorem}

The situation we are interested in is that of a holomorphic submersion $X\to B$, with a relatively ample line bundle $L$, such that $(X_b,L_b)$ admits a cscK metric for each $b\in B$. Thus by the above, each $X_b$ is Chow stable when embedded in projective space using global sections of $k L_b$ for some $k=k(b)$ which \emph{a priori} depends on $b$. We claim that in fact $k$ can be chosen independently of $b$ (which is well known to experts). This follows since $k$ is a function of the ``geometry bounds'' of $(X_b,L_b)$ which are bounded independently of $b$ since $B$ is compact and the cscK metrics and complex structures vary smoothly among the fibres (see \cite{donaldson-projective-embeddings} for the notion of $R$-bounded geometry which we are using). 

\begin{corollary} There exists a projective space $\pr^m$ such that each $X_b$ is embedded in $\pr^m$ using global sections of $kL_b$ with $X_b$  Chow stable for all $b\in B$. \end{corollary}

We now take the GIT quotient of $\Ch(\pr^m)$ by $\SL(m+1)$ as above to obtain a coarse moduli space $M^{st}$ which parametrises Chow stable varieties in $\pr^m$. The family $X\to B$ then corresponds to a morphism $p: B \to M$. We set $\M = p(B)$ to be the image of $B$ under $p$, which is a variety. 

We wish to endow $\M$ with a Weil-Petersson metric. Since a given fibre $(X_b,L_b)$ may have non-trivial automorphism group, there is no universal family $U\to \M$ in general. Instead of endowing $\M$ with an orbifold Weil-Petersson metric, analogous to Fujiki-Schumacher's construction, we give a simpler and more direct argument as follows. We first replace $\M$ with the image $p(B)$ of $B$ in $\M$. While in general there is no universal family over $\M$, by a result of Koll\'ar \cite[Section 2]{kollar-projectivity} there \emph{is} a universal family after passing to a finite cover $\M' \to \M$. Denote by $U'\to \M' $ this universal family. Since $\M'\to \M$ is finite, there are codimension one subschemes $D'$ and $D$ of $\M'$ and $\M$ respectively, such that $M' \to M$ is \'etale. Let $B^o$ be the Zariski open locus of $B$ which maps to $\M'$. 

The universal family $U'\to \M'$ induces a Weil-Petersson metric on $M'$ by Remark \ref{submersion-of-spaces}, which is K\"ahler on $\M'\backslash D'$. Picking a connected component $Z$ of $\M'\backslash D'$, the map $p: B \to \M$ induces a map $p: B^0 \to Z$. The Weil-Petersson metric is now a positive K\"ahler metric on $Z$, and by the construction of Section \ref{sec:prelims}, extends to a smooth $(1,1)$-form on $B$. Of course, one can think of $Z$ as a subvariety of $\M$.

This is all we need to obtain the main result of this section.

\begin{proposition} \label{prop:fibre-automs} With all notation as above, $\ker \scL_{\alpha}$ corresponds to the holomorphic vector fields on $B$ which lift to holomorphic vector fields on $X$, i.e. induce elements of $\mfh$. In particular there is a natural identification $\mfh \cong \mfp$, and the extremal vector field on $B$ lifts to $X$ if and only if its flow preserves the map, or equivalently is an element of $ \mfp$.

\end{proposition}

\begin{proof}  We first show that elements $g \in\Aut(p)$ lift to an automorphism of $X$. Just as above, we obtain a map on a Zariski open locus $B^0$ of $B$ to a moduli space $\M'$ with a universal family $\scU \to \M'$. We begin with the case that $B^0 = B$, so that we may assume $\M'=\M.$ Then the moduli map $$p: B \to \M$$ satisfies $p\circ g = p$, by definition of an automorphism of $p$. Since $\M$ is a \emph{fine} moduli space, the family $X\to B$ is isomorphic to the pullback family $p^*\scU\to B$. We also obtain a pullback family $g^*X \to B$ by the fibre product construction, so that the following diagram commutes: 

\[ \begin{tikzcd}
g^*X \arrow{r}{g_X} \arrow[swap]{d} & X \arrow{d}{\pi} \\%
B \arrow{r}{g}& B
\end{tikzcd}
\] But since $p \circ g = p,$ it follows that $g^*p^*\scU = p^*\scU$, so that $g_X$ is an isomorphism of $X$ covering $\pi$.

In the general case, we still obtain a fibre product morphism $g_X: g^*X \to X$ constructed in the same manner as above. On $B^0$, we obtain that $\pi^{-1}(B^0) \cong p|_{B^0}^*\scU$. Similarly, after perhaps replacing $B^0$ with a Zariski open subset, we obtain that $g_X^*X \to B^0$ is given as the pullback $g^*p^*\scU = p^*\scU$. Thus the morphism $g_X: g^*X \to X$ is an isomorphism on $B^0$, so in particular it is a birational morphism. The same construction gives a birational morphism $(g^{-1})_X: X \to g^*X$, with $(g^{-1})_X \circ g_X = \Id$ on $B^0$, perhaps again after replacing $B^0$ with a Zariski open subset. In summary, we have produced a birational morphisms $g_X: g^*X \to X$ and $(g^{-1})_X: X \to g^*X$, which are isomorphisms over $B^0$. But this is impossible unless these maps are isomorphisms by elementary reasoning: a morphism which is birational must extend to a codimension two locus in the target, and strictly codimension one in the domain. More precisely, consider the map $g_X \circ (g^{-1})_X: X \to X$. Then $g_X \circ (g^{-1})_X$ must be the identity, as it is the identity on a dense set in $X$. Hence if $g_X$ is an isomorphism away from a codimension at least two subscheme $Z$ of $X$, the preimage $g_X^{-1}(Z)$ must have codimension one, hence $(g_X \circ (g^{-1})_X)^{-1}(Z)$ must have codimension one, which contradicts that $g_X \circ (g^{-1})_X$ is an isomorphism. 

Thus in both cases, automorphisms $g \in\Aut(p)$ lift to an automorphisms of $X$. It follows that elements of the Lie algebra $\mfp = \Lie(\Aut(p))$ give holomorphic vector fields on $X$, and moreover if such a vector field has a zero on $B$ it certainly has a zero on $X$. This gives a natural inclusion $\mfp \subset \mfh$. 

We next return to the short exact sequence $$0 \to \scV \to TX \to \pi^*TB \to 0,$$ with $\scV$ the vertical tangent bundle. A global holomorphic vector field $v$ on $X$ which has a zero and is in the image of $H^0(X,\scV)$ in the long exact sequence $$0 \to H^0(X,\scV) \hookrightarrow H^0(X,TX) \to H^0(X,\pi^*TB) \to \hdots$$ must have a zero on some fibre. We have assumed that the polarised automorphism group $\Aut(X_b,H_b)$ of each fibre $(X_b,H_b)$ is discrete, so there are no such vector fields. Thus we obtain an inclusion $\mfh \hookrightarrow H^0(X,\pi^*TB) \cong H^0(B,TB)$, which then sends a vector field with a zero to a vector field with a zero. Hence this provides an isomorphism $\mfp \cong \mfh$.

Finally,  Remark \ref{general-lift} states that as $\alpha$ is a smooth form on $B$ which is pulled back from a K\"ahler metric on $Z \subset M$, the kernel of $\scL_{\alpha}$ can be identified with the holomorphic vector fields whose flow preserves the map $p$. This gives the sequence of isomorphisms $\ker \scL_{\alpha} \cong \mfp \cong \mfh$, as required.
\end{proof}

\begin{remark}\label{lifting-remark-bundle}It is crucial in the above that we assume the automorphism group of the fibres is discrete. In general, when the fibres have continuous automorphisms, it is not the case that all automorphisms of the map to the moduli space lift to the total space of fibrations. As an example, consider a vector bundle $E \to (B,L)$ with induced projective bundle $\pi: \pr(E) \to (B,L)$. Then all automorphisms of $(B,L)$ preserve the moduli map, as the moduli space is just a point. On the other hand, $\Aut(B,L)$ lifts to an action on $(B,L)$ if and only if $E$ is $\Aut(B,L)$-equivariant, which is not always the case. \end{remark}


\begin{remark}\label{rmk:second-uniqueness} Using this we can clarify the uniqueness statement in Fine's result \cite{fine1}, as promised in Remark \ref{rmk:uniqueness}. Indeed, a twisted cscK metric on $p: B\to \M$ is unique if and only if $\Aut(p)$ is discrete. But this is equivalent to $\Aut(X,rL+H)$ being discrete by Proposition \ref{prop:fibre-automs}.

In particular, when the map has no automorphisms, $\ker \scL_{\alpha}$ is invertible modulo constants. In \cite[Theorem 8.1]{fine1}, Fine assumes that the solution of the twisted cscK equation is unique  in its cohomology class, and claims as a general principle that if a solution to a nonlinear PDE with elliptic linearisation is unique, then its linearisation has to be invertible modulo constants \cite[p431]{fine1}. This is not actually true, and so for Fine's proof to hold, the necessary assumption is that the linearisation is invertible modulo constants. It follows from our results above that that uniqueness of solutions of the twisted cscK equation does actually imply invertibility of the linearisation, through our geometric characterisation of uniqueness and our calculation of $\ker \scL_{\alpha}$, and thus Fine's invertibility assumption does actually hold when the twisted cscK metric is unique. When the base of the fibration has complex dimension one, Fine uses a different, valid argument to obtain the necessary invertibility \cite[Lemma 3.12]{fine1}. We thank J. Fine for advice on this point.

 \end{remark}

Although the following two results are not needed in the remainder of the present work (and are possibly known to some experts), we feel they are worth describing since they are rather fundamental to the moduli theory of manifolds admitting cscK metrics and are simple consequences of the approach we have taken. The first result is first due to Fujiki-Schumacher, whose proof uses positivity of the Weil-Petersson metric on the CM-line bundle \cite[Theorem 6.6]{fujiki-schumacher}. 

\begin{theorem} Any proper analytic subspace of the moduli space of polarised manifolds with discrete automorphism group admitting cscK metrics is projective. \end{theorem}

\begin{proof}Let $D$ be such a space. We can assume $D$ is reduced and irreducible since this does not affect projectivity. Again as taking a finite cover does not affect projectivity, we can assume $D$ is normal and admits a universal family. Then the fibres satisfy the ``geometry bounds'' required to apply Donaldson's theorem \cite{donaldson-projective-embeddings}, and realise $D$ as a subvariety of a moduli space of Chow stable varieties just as we did above. For this, one can take a resolution of singularities $D'\to D$, which still admits a universal family, but for this universal family it is clear that the complex structures and cscK metrics vary smoothly, and compactness gives the desired geometry bounds. But then since $D$ is a subvariety of a moduli space constructed as above, the ample line bundle on the Chow variety descends to one on the moduli space, and hence restricts to an ample line bundle on $D$. Thus $D$ is projective, as required.
\end{proof}

The second is the separatedness of the moduli space of polarised manifolds with discrete automorphism group admitting cscK metrics, again originally due to Fujiki-Schumacher \cite[Theorem 6.3]{fujiki-schumacher}, whose proof uses the unique extension property of cscK metrics.

\begin{theorem}
The moduli space of polarised manifolds with discrete automorphism group admitting cscK metrics is separated.\end{theorem}

\begin{proof} Consider a pointed curve $C^o = C\backslash \{p\}$ with a family $(X^o,L^o) \to C^o$ of polarised manifolds admitting cscK metrics.  Separatedness means that there is at most one extension to a family $(X,L) \to C$ with the extension cscK with discrete automorphism group. Suppose not, so that one had two such extensions $(X,L)$ and $(X',L')$. Then by Donaldson's result and the arguments above, these would induce maps to an appropriate moduli space of Chow stable varieties. But the machinery of GIT ensures that a moduli space of Chow stable varieties is automatically separated, giving a contradiction. \end{proof}

\section{Isometries of twisted extremal metrics on maps}\label{sec:invariance}

Here we establish that if $p: (B,L)\to(\M,H)$ admits a twisted extremal metric $\omega$, then the metric is invariant under a maximal compact subgroup of $\Aut(p)$. This is a consequence of the structure of the automorphism group of $\Aut(B,L)$, analogous to the structure of the automorphism group of a manifold admitting an extremal metric due to Calabi. 

Recall that $\Aut(p)$ consists of automorphisms of $p: B\to \M$ which lift to $L$. These are precisely those which have a zero somewhere, and hence also lift to any other line bundle on $B$. Let $\mathfrak{p}$ denote the Lie algebra of $\Aut(p)$, so that $\mfp$ consists of holomorphic vector fields whose flow fixes $p$ and which vanish somewhere. 

Fixing a metric $\omega$ with induced Riemannian metric $g$, as in Section \ref{sec:lifting} each element $\mu \in\mfp$ can be written as $\sum_j \mu^j \frac{\partial}{\partial z^j}$, where $$\mu^j = g^{j\bar k} \partial_{\bar k} f,$$ for some $f \in C^{\infty}(X,\C)$, unique up to the addition of a constant. The isomorphism $T^{1,0}X \cong TX$ between the holomorphic and real tangent bundles gives a correspondence $$g^{j\bar k} \partial_{\bar k} f \to \frac{1}{2}(\nabla u + J \nabla v),$$ where $f = u+iv$ is the decomposition of $f$ into real and imaginary parts. Recall that $\mft\subset \mfp$ denotes the vector fields in $\mfp$ which correspond to Killing vector fields under this isomorphism, which by above are vector fields with purely imaginary holomorphy potential.

In order to state the structure theorem, we suppose $B$ admits a twisted extremal metric, giving an extremal vector field $\nu$.  Denote by $\mfp^{\lambda}$ the elements of $\mfp$ which satisfy $\scL_{\nu} \mu = \lambda \mu$. The argument is a variant of the classical argument of Lichnerowicz, and especially of Calabi's extension to extremal metrics \cite[Theorem 1.1]{calabi-2}.

\begin{theorem}\label{thm:reductivity} 
Suppose $p$ admits a twisted extremal metric. Then one has a decomposition $$\mfp \cong \mft \oplus i\mft \oplus_{\lambda>0} \mfp^{\lambda},$$ and moreover if $\mfp_{\nu}$ denotes the elements of $\mfp$ which commute with $\nu$ then $ \mft \oplus i\mft  \cong \mfh_{\nu}.$ In particular if $\omega$ is a twisted cscK metric, then $\mfp$ is reductive.

 \end{theorem}

\begin{proof} Recall from Section \ref{sec:lifting} the operator \begin{equation}\scL_{\alpha}(\phi) = - \mathcal{D}^*_{\omega} \mathcal{D}_{\omega} (\phi) + \frac{1}{2} \langle \nabla \Lambda_{\omega} \alpha,  \nabla \phi \rangle  +  \langle i \partial \overline{\partial} \phi, \alpha \rangle,\end{equation} which was shown to induce an isomorphism  \begin{align*}\ker_0 \scL_{\alpha} &\to \mfp, \\ f  &\to g^{j\bar k} \partial_{\bar k} f\end{align*} where $\ker_0$ denotes the functions with mean value zero with respect to the volume form induced by $\omega$. $\scL_{\alpha}$ is simply the linearisation of the twisted scalar curvature at a twisted cscK metric.

Suppose first that $\omega$ is a twisted cscK metric. In coordinates Lichnerowicz operator takes the form $$\mathcal{D}^*_{\omega} \mathcal{D}_{\omega} (\phi) = \Delta^2 \phi + R^{\bar k j} \nabla_j \nabla_{\bar k} \phi + g^{j\bar k}\nabla_j S(\omega)\nabla_{\bar k} \phi.$$ In general this is not a real operator due to the third term. This implies  \begin{equation}\label{explicitG}\scL_{\alpha}(\phi) =(\Delta^2 \phi + R^{\bar k j} \nabla_j \nabla_{\bar k} \psi  -  \langle i \partial \overline{\partial} \phi, \alpha \rangle) + g^{j\bar k}\nabla_j (S(\omega) - \Lambda_{\omega}\alpha)\nabla_{\bar k} \phi.\end{equation} Since the last term then vanishes, this is a real operator when $\omega$ is twisted cscK. Thus $u + iv \in \ker \scL_{\alpha}$ for real functions $u,v$ if and only if $u,v \in \ker \scL_{\alpha}$. Since a vector field $\mu$ corresponds to a Killing vector field under the isomorphism $T^{1,0}X \cong TX$  if and only if its holomorphy potential is purely imaginary, this implies that $\ker_0 \scL_{\alpha}$ induces an isomorphism $\mfp \cong \mft \oplus i\mft$, which means that $\mfp$ is reductive.

In the general case $\scL_{\alpha}$ is no longer a real operator, and in particular we have $$\scL_{\alpha}(\phi) - \bar \scL_{\alpha}(\phi) = g^{j \bar k} \big( \nabla_j (S(\omega) - \Lambda_{\alpha}\omega)  \nabla_{\bar k} \phi - \nabla_j \phi \nabla_{\bar k} (S(\omega) - \Lambda_{\alpha}\omega) \big).$$ If $\omega$ is a twisted extremal metric,  then $\nu = g^{j\bar k} \partial_{\bar k}( S(\omega) - \Lambda_{\alpha}\omega) \in \mfp$. If $f$ is another holomorphy potential with corresponding vector field $\mu_f$, then the Lie bracket can be computed as $$[\nu, \mu_f] = g^{q\bar r} \nabla_{\bar r}  g^{j \bar k} \left((S(\omega) - \Lambda_{\alpha}\omega)  \nabla_{\bar k}f  - \nabla_j (S(\omega) - \Lambda_{\alpha}\omega) \nabla_{\bar k}f\right).$$ Thus if $[\nu, \mu_f] $, or equivalently the Lie derivative $\scL_{\nu} \mu_f = 0$, then $\scL_{\alpha}(\phi) = \bar \scL_{\alpha}(\phi)$. Denote by $\mfp_{\kappa}$ the elements of $\mfp$ which commute with $\nu$. Since elements of $\mft$ automatically commute with $\nu$ as they correspond to Killing vector fields, as above we obtain a decomposition $\mfp_{\nu} \cong \mft \oplus i\mft.$

The conjugate $\bar\scL_{\alpha}$ acts on $\ker_0 \scL_{\alpha}$ since $\bar\scL_{\alpha}$ and $\scL_{\alpha}$ commute by equation \eqref{explicitG}. The spectrum of $\bar\scL_{\alpha}$ is non-negative, since if $\phi$ is an eigenfunction, $ \lambda \int_{B} |\phi|^2\omega^n = \int_B \bar{\phi} \scL_{\alpha} \phi\omega^n \geq 0$ by equation \eqref{eq:needs-conj}. The space $\ker_0 \scL_{\alpha}$ thus splits as $$\ker_0 \scL_{\alpha} \cong \oplus_{\lambda\geq 0} E_{\lambda},$$ where $E_{\lambda}$ denotes the eigenspace of eigenvalue $\lambda$. If $\phi$ is an eigenfunction then $$\lambda \phi = \bar\scL_{\alpha} \phi = (\bar\scL_{\alpha}  - \scL_{\alpha})\phi.$$ But $$ g^{j\bar k} \partial_{\bar k} (\bar\scL_{\alpha}  - \scL_{\alpha})\phi = -[\nu, \mu_{\phi}],$$ where $ \mu_{\phi}$ is the vector field corresponding to $\phi$ as above. In particular, the eigenspaces can be characterised as satisfying $$\lambda \mu_{\phi} = -[\nu, \mu_{\phi}],$$ thus there is a splitting $$\mfp \cong\mfp_{\nu}\oplus_{\lambda>0} \mfp^{\lambda}$$ as required.
\end{proof}

The proof of the automorphism invariance of a twisted extremal metric is now a consequence of the above, proven in the same way as Calabi's proof for extremal metrics \cite[Theorem 1.3]{calabi-2}. We argue exactly as in  Gauduchon's slightly different proof \cite[Theorem 3.5.1]{gaud-book} of the same statement. Denote by $\scH$ the subgroup of the $\Aut_0(p)$ consisting of isometries of $\omega$, where $\Aut_0(p)$ is the identity component, so that $\mft$ is the Lie algebra of $\scH$.

\begin{corollary}\label{thm:isometries} The isometry group $\scH$ is a maximal, compact, connected subgroup of $\Aut(p)$. \end{corollary}

\begin{proof} Suppose not, and let $\tilde \mft$ be the Lie algebra of a larger compact subgroup of $\Aut(p)$. Let $\mu$ be a vector field in $\tilde \mft \backslash \mft$, and write $\mu = \mu_0 + \sum_{\lambda>0} \mu_{\lambda}$, so that $\scL_{\nu} \mu_{\lambda} = -\sum_r \lambda^r \mu_{\lambda}.$ It is clear that it is enough to assume $\mu=\mu_{\lambda}$ for some $\lambda$. 

First suppose $\mu \in \mfp^{\lambda}$. If $\scB$ is the Killing form of $\tilde \mft$, then the kernel of $B$ coincides with the centre of $\tilde \mft$ (see e.g.  \cite[Proof of Theorem 3.5.1]{gaud-book}). We claim that $\mu \in \ker \scB$. Recall that $\scB(\mu, \zeta_1) = \tr (ad_{\mu} \circ ad_{\zeta_1})$. Let $\zeta_1,\zeta_2$ be of degree $\lambda_1, \lambda_2 \geq 0$. Then $[\mu,[\zeta_1,\zeta_2]]$ is of degree $\lambda+\lambda_1+\lambda_2$. As $\lambda > 0$ and $\lambda_1+\lambda_2\geq 0$,  $ad_{\mu} \circ ad_{\zeta_1}$ is a linear operator on a graded vector space $V = \oplus_k V_k$ which sends $V_k$ to $\oplus_{l \neq k} V_l$ (in our situation, one can even take $l>k$). In particular, its trace has to be $0$, and hence $B(\mu, \zeta_1) = 0$. Thus $\mu$ is in the centre of $\tilde \mft$, contradicting the fact that $\mu$ does not commute with $\nu$ since $\lambda>0$.

Thus we can assume $\mu \in \mfp^0$, and hence can be taken to lie in $ i \mft$ since $\mu$ is not element of $\mft$ by hypothesis. Then $\nu = \nabla f$ for some $f$. But the flow of a complete vector field cannot be simultaneously contained in a compact subgroup of the diffeomorphism group of $X$ and be the gradient of a function \cite[Proposition 3.5.1]{gaud-book}. This gives a contradiction, and proves the result.
\end{proof}

We will also need the following result, which follows from a similar analysis as Theorem \ref{thm:reductivity}.

\begin{proposition}\label{torus-lemma}Suppose $\omega_B$ is a twisted extremal metric, which by above is invariant under a maximal torus $T$ of the automorphism group of $p$. Then the operator $$ \scL'_{\alpha}:  C^{4,\alpha} (B, \mathbb{R})^T \times \overline{\mathfrak{t}}\rightarrow C^{0, \alpha} (B, \mathbb{R})^T$$ defined by $$\scL'_{\alpha}(\phi,f) = \scL_{\alpha} - f.$$ Then $\scL'_{\alpha}$ is well-defined and surjective.  Here, the notation $C^{k,\alpha} (B, \mathbb{R})^T$ means real valued $T$-invariant functions on $B$, and $\overline{\mathfrak{t}}$ consists of the functions $f$ such that $v_f=J \nabla f$ is a Killing vector field with zeros on $B$, where the gradient is taken with respect to the metric on $B$.  \end{proposition}

\begin{proof} By well-defined, we mean that $\scL'_{\alpha}$ sends real valued torus invariant functions to real valued torus invariant functions. The torus invariance is an obvious consequence of all the data involved being invariant. To prove that $\scL'_{\alpha}(\phi)$ is real if $\phi$ is real, it is enough to prove the same statement for $\scL_{\alpha}$. 

We argue as in Arezzo-Pacard-Singer \cite[p16]{APS}. Let $v=v_{S(\omega)-\Lambda_{\omega}\alpha}=J\nabla (S(\omega)-\Lambda_{\omega}\alpha)$ denote the real holomorphic vector field induced by the extremal vector field. Using equation \eqref{explicitG}, as in \cite{APS} the operator $\scL_{\alpha}$ can be written as $$ \scL_{\alpha}(\phi) =(\Delta^2 \phi + R^{\bar k j} \nabla_j \nabla_{\bar k} \phi  -  \langle i \partial \overline{\partial} \phi, \alpha \rangle) -Jv(\phi)+iv(\phi).$$ Since $\psi$ is assumed to be invariant under $T$, and hence invariant under $v$, we have $v(\psi)=0$, leaving a real operator as claimed.


Note that the operator $\scL_{\alpha}$ is self-adjoint by equation \eqref{eq:needs-conj}: we have $$\int_B \psi \scL_{\alpha} (\phi) \omega_B^n = \int_B \scL_{\alpha}(\psi)  \phi \omega_B^n.$$ Thus the image of $\scL_{\alpha}$  is orthogonal to its kernel. Since it has Fredholm index zero by self-adjointness, it suffices to show that its kernel is simply $\overline \mft$ to prove the result.

Thus we suppose $\scL_{\alpha}(\phi)=0$, with $\phi$ a real function. Then by Proposition \ref{prop:fibre-automs}, we know that $\phi$ is a holomorphy potential for some holomorphic vector field which preserves the map. Since $\phi$ is a \emph{real} holomorphy potential, this vector field must be $J$ of a Killing field, proving the result. 

 \end{proof}
 
 \begin{remark}  In fact the observation that $\scL_{\alpha}$ is real valued is a rather general moment map phenomenon, related to work of Kirwan in the finite dimensional GIT setting \cite[Corollary 6.11]{kirwan}. Without restricting to torus invariant functions, the above Proposition has no hope of being true in general. This is the reason we were required to prove a twisted extremal metric is invariant under a maximal torus, as this allows us work with torus invariant functions above.
 
 \end{remark}

\section{Approximate solutions}\label{sec:approx}

The goal of this section is to construct an approximate extremal metric on $X$ in the class $H+rL$ using an inductive approach. Recall that $H$ is a line bundle on $X$ which is relatively ample over $B$ and such that $\Aut(X_b)$ is discrete for all $b\in B$, and $H$ is an ample line bundle on $B$. Moreover we assume that $\omega_0$ restricts to a cscK metric on each $(X_b,H_b)$ and that $\omega_B \in c_1(H)$ is a twisted extremal metric on $B$, where the twist is given by the Weil-Petersson metric induced by the cscK metrics on the fibres of the map $X\to B$. We also assume the extremal vector field on $B$ lifts, which simply means that it is an element of $\mfp$. For our inductive approach, we fix $r$ and begin with the K\"ahler metric $\omega_r = \omega_0 + r\omega_B$ on $X$. By taking $r$ large enough, we can assume $\omega_0$ is actually ample.  Recall that $\omega_B$ is invariant under a maximal (compact) torus of automorphisms $T\subset \Aut(p)$. 

\begin{theorem}\label{thm:approx-soln} In what follows, all data is taken to be invariant under the natural lift of the torus $T$ to $X$ described in Section \ref{sec:lifting}. For all $p \geq 0$, there exist functions $f_0,\hdots, f_{p-1} \in C^{\infty}(B),\   l_0,\hdots,l_p \in C^{\infty}_0(X)$ with $$\varphi_p = \sum_{i=0}^{p-1} f_ir^{-i+1}, \quad \lambda_p 
= \sum_{i=0}^p l_ir^{-i}$$ holomorphy potentials $\beta_p = \sum_{i=0}^p b_ir^{-i}$ with respect to  $\omega_r$ such that the K\"ahler metric $$\omega_{r,p} = \omega_r + \ddb \varphi_p+\ddb \lambda_p$$ satisfies
$$ S(\omega_{r,p}) - \frac{1}{2}\langle \nabla_{\omega_r} \beta_p,\nabla_{\omega_r} (\varphi_p +\lambda_p)\rangle - \beta_p = O(r^{-p-1}).$$ \end{theorem}

We prove this by induction on $p$. The main difficulty in constructing such a metric is controlling the linearisation of the scalar curvature and the twisted extremal operator. For the scalar curvature we can then apply the results of Fine \cite{fine1}. For the twisted extremal operator we apply the results of Section \ref{sec:lifting}, which allow us to identify the kernel of the linearisation. Another new difficulty occurring compared to Fine's work arises due to the fact are solving the \emph{extremal} equation rather than the cscK equation, leading to the inner products of gradients in the equation we wish to solve. In fact this is relatively straightforward to deal with, as these terms essentially do not affect the inductive step.

To handle the linearisation of the scalar curvature on the fibres, we introduce some notation. The form $\omega_0$ restricts to a K\"ahler metric on each fibre $X_b$, and hence one can define an operator $\scL_b$ to be the linearisation of the map $\Phi \to S(\omega_b + \ddb \Phi)$ for $\Phi \in C^{\infty}(X)$. These glue to form an operator $\scL_0$ which restricts to $\scL_b$ on $X_b$ for all $b\in B$. 

The relevant results from Fine's work that we shall appeal to are then the following:

\begin{proposition}\cite[Section 3]{fine1}\label{fine-prop}
\begin{enumerate}[(i)]
\item The scalar curvature of $\omega_r$ is given by $$S(\omega_r) = S(\omega_b) + r^{-1} \theta + O(r^{-2}),$$ where $\theta \in C^{\infty}(X)$ satisfies  $\int_{X/B} \theta\omega_0^m = S(\omega_B) - \Lambda_{\omega_B} \alpha.$
\item The linearisation of the fibrewise scalar curvature at $\omega_r$ is given by  $$\scL_r(\phi) =\scL_0(\phi) + O(r^{-1}).$$ 
\item Let $\Theta \in C^{\infty}_0(X)$. Then there exists a unique solution $\psi \in C^{\infty}_0(X)$ of $$ \scL_0(\psi) = \Theta.$$ If moreover $\omega_B$ is invariant under some maximal torus $T\subset \Aut(p)$, then torus invariance of $\psi$ implies torus invariance of $\Theta$.
\end{enumerate} 
\end{proposition}

\begin{proof} These are all proven in \cite{fine1} except the automorphism invariance, which is a simple corollary of Fine's construction. \end{proof}

This provides the zeroth step of our argument, precisely because we are assuming the extremal vector field on $B$ lifts, i.e. is an element of $\mfp$.

\begin{corollary}
Theorem \ref{thm:approx-soln} holds for $p=0$.
\end{corollary}

\begin{proof} Set $l_0 = 0$ and $b_0 = S(\omega_b)$. There are no $f_i$-terms in this case. \end{proof}

Note that the scalar curvature $S(\omega_b)$ is a constant independent of $b$ by Remark \ref{fibre-independence}. The result for $p=1$ follows from the following simple lemma, analogous to \cite[Lemma 4.12]{lu-seyyedali}. 

\begin{lemma}\label{lem:subdominant}
Let $\eta\in  C^{\infty}(B), \psi \in C^{\infty}(X)$. Then $$\langle \nabla_{\omega_r} \eta,\nabla_{\omega_r} \psi \rangle = O(r^{-1}).$$ If moreover $\psi \in C^{\infty}(B)$, then $$\langle \nabla_{\omega_r} \eta,\nabla_{\omega_r} \psi \rangle =r^{-1}\langle \nabla_{\omega_B} \eta,\nabla_{\omega_B} \psi \rangle +  O(r^{-2}).$$
\end{lemma}

\begin{corollary}
Theorem \ref{thm:approx-soln} holds for $p=1.$
\end{corollary}

\begin{proof} Applying Proposition \ref{fine-prop} $(iii)$ to $\int_{X/B} \theta \omega_0^m -\theta$ gives a function $l_1 \in C^{\infty}_0(X)$. Let $b_1$ be the holomorphy potential of the extremal vector field $\nu$ on $B$, i.e. $$S(\omega_B) - \Lambda_{\omega_B} \alpha = b_1,$$  so that  $\beta_1 = r^{-1}b_1$. Then the result follows from Lemma \ref{lem:subdominant} with $f_0=0$. 
\end{proof}

For $p=2$ we begin the inductive argument. The new difficulty is that we are no longer at a \emph{genuine} cscK metric, and so one needs to vary the arguments. This requires some new techniques and new notation, for which we follow \cite[Section 3.3]{fine1}.

Let $\Omega_0 \in c_1(L)$ be relatively K\"ahler metric, with restriction $\Omega_b$ to a fibre, and let $\Omega_B \in c_1(H)$ be a K\"ahler metric on $B$. Denote by $\scL_{\Omega_0}$ the operator which glues the linearisation of the scalar curvature of the restriction $\Omega_b$ for all $b$. Let $\widetilde{\Omega}_r = \Omega_0 + r\Omega_B$, and let $\Omega_r = \widetilde{\Omega}_r + \ddb ( \phi_r )$, where $$\phi_r = \sum_{j=1}^d \psi_j r^{-j},$$ for functions $\psi_j \in C^{\infty}(X)$.  With this notation in place, we require the following, which is very similar to and mostly contained in the analogous work of Fine \cite[Section 3.3]{fine1}. 

\begin{proposition}\label{prop:inductive} The linearisations satisfy the following properties.
\begin{enumerate}[(i)]
\item Letting  $\scL_{\Omega_r}$ denote the linearisation of the scalar curvature at $\Omega_r$, we have $$\scL_{\Omega_r} = \scL_{\Omega_0}+ r^{-1}D_1 + r^{-2}D_2 + O(r^{-3}),$$ where $D_1$ only depends on $\psi_1$ and $D_2$  only depends on $\psi_1$ and $\psi_2$.
\item For a function $\pi^* f$ pulled back from $B$ we have  $$ D_1(\pi^* f) = 0, \quad \int_{X/B} D_2( \pi^* f) \Omega_0^m = D_{\Omega_B}(f),$$ where  $$D_{\Omega_{B}} = \scL_{\alpha} + \frac{1}{2} \langle \nabla \big( S(\Omega_B) - \Lambda_{\Omega_B} (\alpha) \big), \nabla \big( \cdot \big) \rangle .$$
\end{enumerate}
\end{proposition}
From Proposition \ref{prop:linearisation-twisted} together with the definition of $\scL_{\alpha}$, the last point states that the base component of $D_2$ is the linearised operator of the twisted extremal operator with respect to the original metric $\Omega_B$ on the base, when making perturbations of the type that we are making. Note that while $D_1,D_2$ themselves depend on some of the $\psi_j$, their fibre integrals do not.

\begin{proof} We begin with the proof in the case that $\psi_j =0$ for $j = 1,\hdots,d$, so that $\Omega_r = \widetilde{\Omega}_r $. Recall from Proposition \ref{fine-prop} that before perturbing the metric, the scalar curvature of $\widetilde{\Omega}_r$ satisfies $$ S \big( \widetilde{\Omega}_r \big) = S(\Omega_b) + r^{-1} \theta + O(r^{-2}),$$ where $\theta \in C^{\infty}(X)$ satisfies  $\int_{X/B} \theta\Omega_0^m = S(\Omega_B) - \Lambda_{\Omega_B} \alpha.$ From this we see that the initial term in the expansion of the linearisation is simply the fibrewise linearisation. This shows the first statement in this case. 

For the second statement, note that for a function $f$ on $B$, we have $$\widetilde{\Omega}_r + \ddb \pi^* f = \Omega_0 + r \big( \pi^* \Omega_B + r^{-1} \ddb \pi^* f \big). $$ Since we are changing the base metric by $ r^{-1} \ddb f$ rather than $\ddb f$, this shows that the first non-zero component in the linearisation applied to such functions is the $D_2$-term. Moreover, the fibre integral of the linearisation of $\theta$ will then be the linearisation of the operator $$ f \mapsto S(\Omega_B + r^{-1} \ddb f ) - \Lambda_{\Omega_B+r^{-1} \ddb f} \alpha,$$ showing that the fibre integral of the $D_2$-operator is the operator $D_{\Omega_B}$, which proves the second claim in the case $\Omega_r = \widetilde{\Omega}_r $.

We now proceed to the general case. When perturbing the metric by $\phi_r$, the expansion of the scalar curvature changes, but we then use the linearisation as above to identify the new terms. There is no change in the initial term, since $\phi_r$ is of order $r^{-1}.$ Moreover, the change in the $r^{-1}$-term is $\mathcal{L}_0 (\psi_1)$. The linearisation of this term in the direction of a function pulled back from the base is $0$. This gives the statement about the $D_1$ operator. Similarly, there will be a change in the scalar curvature (and consequently the linearised operator) in the $r^{-2}$-term, which now involves both $\psi_1$ and $\psi_2$. However, the change in the fibre integral of the scalar curvature coming from the $C^{\infty}(B)$ component $\pi^*\Phi_r$ of $\phi_r $ will be of order $r^{-3}.$ This is because for this component, the change in fibre integral of the scalar curvature is $$r^{-2} D_{\Omega_B} ( \Phi_r ) + O(r^{-4}) = r^{-3} D_{\Omega_B} (  \Psi_1 ) + O(r^{-4}).$$ Here $\Psi_1 \in C^{\infty}( B,\mathbb{R})$ is the base component of $\psi_1$, namely its fibre integral. Thus there will be no change in the fibre integral of the operator $D_2$. This concludes the proof of the general case.
\end{proof}

This is all that is needed to construct the approximate solution for all $p$. 

\begin{proof}[Proof of Theorem \ref{thm:approx-soln}] Torus invariance is automatic as all of the initial data is torus invariant.

We suppose that an approximate solution has been produced for $p-1$, so we have $\varphi_{p-1},\lambda_{p-1}$ and holomorphy potentials $\beta_{p-1}$ such that the K\"ahler metric \begin{align*} \omega_{r,p-1} = \omega_{r}+ \ddb \varphi_{p-1}+\ddb \lambda_{p-1} = \omega_{r} +  \ddb \bigg( \sum_{i=0}^{p-2} f_ir^{-i+1}+ \sum_{i=0}^{p-1} l_ir^{-i} \bigg) \end{align*} satisfies
\begin{equation}\label{eq:pthstep} S(\omega_{r,p-1}) - \frac{1}{2}\langle \nabla_{\omega_{r}} \beta_{p-1},\nabla_{\omega_{r}} (\varphi_{p-1} +\lambda_{p-1})\rangle - \beta_{p-1} = O(r^{-p}).\end{equation}

For $f_{p-1} \in C^{\infty}(B)$, we have $$ S(\omega_{r,p-1} + \ddb f_{p-1} r^{-p+2}) =  S(\omega_{r,p-1}) + \left(\int_{X/B} D_2(f_{p-1}) \omega_{r,p-1}^m +\Theta_p\right)r^{-p} + O(r^{-p-1}),$$ where $\Theta_p\in C_0^{\infty}(X)$, with the fibre integral measured with respect to $\omega_{r,p-1}$. Noting $\omega_{r,p-1} = \omega_r + O(r^{-1})$, and using Proposition \ref{prop:inductive} to relate the fibre integral to $\scL_{\alpha}$, up to terms of $O(r^{-p-1})$ we have $$ S(\omega_{r,p-1} + \ddb f_{p-1} r^{-p+2}) =  S(\omega_{r,p-1}) + (\scL_{\alpha} f_{p-1} + \frac{1}{2} \langle \nabla b_1 , \nabla f_{p-1}  \rangle + \Theta'_p)r^{-p} + O(r^{-p-1}),$$ where $\Theta'_p \in C_0^{\infty}(X)$ measured with respect to $\omega_r$. This follows from Proposition \ref{prop:inductive} and the fact that $b_1 = S(\omega_B) - \Lambda_{\omega_B} (\alpha).$

Let $w_p$ be the $O(r^{-p})$ term of equation \eqref{eq:pthstep}. Write $w_p = \Psi_p + \Phi_p$, where $\Psi_p \in C^{\infty}(B)$ and $\Phi_p \in C^{\infty}_0(X)$. The operator $\scL'_{\alpha}(\phi,\gamma) = \scL_{\alpha}\phi - \gamma$ , where $\gamma$ is a holomorphy potential, with respect to $\omega_B$, of some holomorphic vector field $\nu_p\in \mfp$, is surjective by Proposition \ref{torus-lemma} (and recalling all data is torus invariant), so we can solve $\scL_{\alpha}f_{p-1} - \tilde b_p = -\Psi_p$ for $\tilde b_p$ a holomorphy potential with respect to $\omega_B$. Now the holomorphy potential for the vector field corresonding to $r^{-1} \nu_p$ on $X$ with respect to $\omega_r = r\omega_B + \omega_0$ is of the form $ b_p = \tilde b_p + O(r^{-1})$. With this choice we therefore have $$S(\omega_{r,p-1} + \ddb f_{p-1} r^{-p+2}) =  S(\omega_{r,p-1}) + \left(b_p + \frac{1}{2} \langle \nabla b_{1},\nabla f_{p-1} \rangle -\Psi_p + \Theta'_p\right)r^{-p} + O(r^{-p-1}).$$ Note that $b_p$ is certainly real valued, and the corresponding vector field certainly vanishes on $X$ since it vanishes on $B$, thus $b_p$ is of the desired form.

Continuing with functions from the base, since $\varphi_p = \varphi_{p-1}+f_{p-1} r^{-p+2}$, we have 
\begin{align*} \langle \nabla_{\omega_r} \beta_p,\nabla_{\omega_r} \big( \varphi_p + \lambda_{p-1} \big) \rangle =& \langle \nabla_{\omega_r} \beta_{p-1},\nabla_{\omega_r} \big( \varphi_{p-1} + \lambda_{p-1} \big)  \rangle +  \langle \nabla_{\omega_r} r^{-p} b_p ,\nabla_{\omega_r} \big( \varphi_{p} + \lambda_{p-1} \big)  \rangle   \\
&+\langle \nabla_{\omega_r} \beta_{p-1},\nabla_{\omega_r} f_{p-1} r^{-p+2} \rangle \\
=&  \langle \nabla_{\omega_r} \beta_{p-1},\nabla_{\omega_r} \big( \varphi_{p-1} + \lambda_{p-1} \big)  \rangle + r^{-p} \langle \nabla_{\omega_B} b_{1},\nabla_{\omega_B} f_{p-1} \rangle_{\omega_B} \\
& + O ( r^{-p-1}),
\end{align*}
where we have used the expansion of $\phi_p, \lambda_{p-1}, \beta_p$ and Lemma \ref{lem:subdominant}.

Finally we choose $l_p$. We have \begin{align*}S(\omega_{r,p-1} +\ddb f_{p-1} r^{-p+2} + & \ddb l_pr^{-p}) =  S(\omega_{r,p-1}) \\ & + \left(b_p + \frac{1}{2} \langle \nabla b_{1},\nabla f_{p-1} \rangle - \Psi_p+ \Theta'_p + \scL_0 l_p\right)r^{-p} + O(r^{-p-1}).\end{align*} We can solve uniquely $\scL_0 l_p = -\Theta_p' - \Phi_p$, giving $$S(\omega_{r,p-1} +\ddb f_{p-1} r^{-p+2} + \ddb l_pr^{-p}) =  S(\omega_{r,p-1}) + \big( b_p + \frac{1}{2} \langle \nabla b_{1},\nabla f_{p-1} \rangle - \Psi_p - \Phi_p \big) r^{-p} + O(r^{-p-1}).$$ This removes the error term up to the change in the holomorphy potential. But by Lemma \ref{lem:subdominant}, we have $$\langle \nabla_{\omega_r} \beta_p,\nabla_{\omega_r} r^{-p} l_p \rangle = O(r^{-p-1})$$ and so no additional error terms to order $r^{-p}$ are caused by the inner products of the new gradient terms. Summing the various terms proves the inductive step, and hence the result. 
\end{proof}

\section{The implicit function theorem argument}\label{sec:ift}

This section proves the main result of the present work, namely a construction of extremal metrics on certain fibrations. We recall the setup. We have a fibration $X\to B$, such that $X$ is endowed a relatively ample line bundle $H$ such that each fibre $(X_b,H_b)$ admits a cscK metric and $X_b$ has discrete automorphism group. We assume $(B,L)$ admits a twisted extremal metric, where the twisting form $\alpha$ is the Weil-Petersson metric induced from the cscK metrics on the fibres $(X_b,H_b)$. We proved in Section {} that in this situation $(X,rL+H)$ admits an ``approximately extremal'' metric. Here we prove the existence of a genuine solution of the extremal equation:

\begin{theorem}\label{mainthm-sec6}
$(X,rL+ H)$ admits an extremal metric for all $r\gg 0$.
\end{theorem}

Fine's work uses a quantitative inverse function theorem to perturb from an ``approximate cscK metric'', analogous to the metrics we have constructed in Section \ref{sec:approx}, to a genuine cscK metric \cite{fine1}. Since the operator we consider only has a one-sided inverse, we shall instead rely on a quantitive implicit function theorem as in \cite[Theorem 29]{bronnle}.

\begin{theorem}\label{implicitfnthm}\cite{bronnle} Consider a differentiable map of Banach spaces $F: B_1 \to B_2$  whose derivative at $0$ is surjective with right-inverse $P$. Denote by 
\begin{enumerate}[(i)]
\item $\delta'$ the radius of the closed ball in $B_1$ centred at $0$ on which $F-DF$ is Lipschitz of constant $(2\|P\|)^{-1}$,
\item $\delta = \delta' (2\|P\|)^{-1}$.
\end{enumerate} 
For all $y \in B_2$ such that $\|y-F(0)\| < \delta$, there exists $x\in B_1$ satisfying $F(x) = y$. 
\end{theorem}

We shall apply this with $F$  given by the extremal operator. Thus the main point is to control both the non-linear terms of the extremal operator and to bound the one-sided inverse of (a variant of) the linearisation. For this, we apply some estimates due to Fine, as well as some estimates specific to our more general situation.  

The result of Fine which we shall use is the following Schauder-type estimate. 

\begin{lemma}\label{globalschauder}\cite[Lemma 6.8]{fine1} Denote by $g_r$ the Riemannian metric corresponding to the K\"ahler form $\omega_{r,p}$. For each $k$ and $p$, there are constants $C, r_0 >0$ such that for all $r \geq r_0$ \begin{align*} \| \phi \|_{L^{2}_{k+4} (g_r) } \leq C \big( \| \phi \|_{L^{2} (g_r)} + \|  \mathcal{D}^* \mathcal{D}(\phi) \|_{L^2_k (g_r)} \big),
\end{align*}
for all $\phi \in L^{2}_{k+4}$.  Here $\mathcal{D}^* \mathcal{D}$ denotes the Lichnerowicz operator $\mathcal{D}_{\omega_{r,p}}^* \mathcal{D}_{\omega_{r,p}}$ defined using $\omega_{r,p}$.
\end{lemma}
Fine proves this result for general fibrations $X\to B$ with metrics of the form we have considered. We emphasise that the results of \cite[Section 5]{fine1}, on which the previous lemma is based, do not use any special metric metric on $B$ and so also apply directly to our situation.

While in Fine's situation the operator $\mathcal{D}_{\omega_{r,p}}^* \mathcal{D}_{\omega_{r,p}}$ is invertible when restricting to functions of mean value zero, this is no longer the case for us due to the presence of automorphisms. In order to state the result we require in our situation, we begin by describing the change in potentials for holomorphic vector fields on $X$, as this space corresponds to the kernel and cokernel of the Lichnerowicz operator. 

As in Section \ref{sec:approx} we can assume $\omega_0$ is actually positive. Let $\xi$ be a holomorphic vector field on $X$. Let $h_X$ be its potential with respect to our initial metric $\omega_0$ on $X$. Recall from Propositions \ref{automs-prop} and \ref{prop:fibre-automs} that $\xi$ corresponds uniquely to a holomorphic vector field $\widetilde{\xi}$ on $B$. Let $h_B$ be the holomorphy potential of this metric with respect to $\omega_B$, the twisted extremal metric on $B$. Then $\iota_{\xi} \pi^* \omega_B = \pi^* (\iota_{\widetilde{\xi}} \omega_B ) = \pi^* \bar \partial h_B = \bar \partial (\pi^* h_B)$. Thus $r \pi^* h_B + h_X  $ is a holomorphy potential for the vector field $\xi$ on $X$ with respect to the metric $\omega_{r,0} = \omega_X + r \pi^* \omega_B$.

Let $\varphi_{r,p}$ be the K\"ahler potential of $\omega_{r,p}$ with respect to $\omega_{r,0}$ as in the construction of the approximate solutions in Theorem \ref{thm:approx-soln}, i.e. $\omega_{r,p} = \omega_{r,0} + \ddb \varphi_{r,p}$. Then as described in Section \ref{subsec:extremal}, the holomorphy potential for $\xi$ with respect to $\omega_{r,p}$ is given by $r \pi^* h_B + h_X + \frac{1}{2} \langle \nabla \varphi_{r,p} , \xi \rangle$. 

\begin{definition}For $h = h_B \in \overline{\mathfrak{p}}$, we define the lift of $h$, denoted by $\tau_{r,p} (h)$, to be the element $$\tau_{r,p} (h) =  r \pi^* h_B +h_X +  \frac{1}{2} \langle \nabla \varphi_{r,p} , \nabla h \rangle \in \overline{\mathfrak{h}}.$$\end{definition} Thus we have encoded the change in holomorphy potential with $(r,p)$ as a sequence of maps $$\tau_{r,p} : \overline{\mathfrak{p}} \rightarrow C^{\infty} (X).$$

It will be necessary to work with functions invariant under the fixed maximal torus $T$, so we consider the space $(L_k^2)^T$ of $T$-invariant functions in $L_k^2$. Here $T$ is a maximal torus in $\textnormal{Aut} (X)$ corresponding to a maximal torus in $\textnormal{Aut(p)}$, the automorphisms of the map. Then the pullback of invariant functions on $B$ give $T$-invariant functions on $X$, and all the approximate metrics $\omega_{r,p}$ are $T$-invariant. We are then also restricting the map $\tau_{r,p}$ to $\overline{\mathfrak{t}}$, the Lie algebra of the maximal torus.

\begin{lemma} The  operator $ (L^2_{k+4})^T \times  \overline{\mathfrak{t}} \to (L^2_{k})^T$ $$ (\psi, h) \mapsto \mathcal{D}^* \mathcal{D} (\psi) + \tau_{r,p}(h) $$ admits a right inverse $Q$ satisfying the bound $$\|Q(\psi)\|_{L^2_ {k+4}(g_r)} \leq K r^3 \|\psi\|_{L^2_k(g_r)}.$$ Here $K$ is a constant independent of $r$.
\end{lemma}

Note that in the statement above, the operator, and hence also its inverse, depends  on both $r$ and $p$. 

\begin{proof}
From the description of the kernel of the Lichnerowicz operator on $X$ (and hence of its cokernel), we get the existence of $Q$. For the bound on the $L^{2}_{k+4}$-factor of $Q(\psi)$, we have to establish such a bound when $\psi$ is assumed to be orthogonal to $\tau_{r,p}(\overline{\mathfrak{t}})$ as this is the image of $\mathcal{D}^* \mathcal{D}$. 

The bound follows as in \cite[Lemma 6.5, 6.6, 6.7]{fine1} which establishes a suitable Poincar\'e inequality for the $\mathcal{D}$. These three results go over exactly as in Fine's work in our setting, except that for the result of \cite[Lemma 6.6]{fine1}, which gives a Poincar\'e inequality for the $\overline{\partial}$-operator on $\Gamma ( TX)$, we have to work orthogonally to the kernel of this operator. This in turn gives the required Poincar\'e inequality for $\mathcal{D}$, where one has to work orthogonally to $\tau_{r,p}(\overline{\mathfrak{t}})$, not just the constants. 

Thus the right-inverse map $Q$ that sends a function in the orthogonal complement to $\tau_{r,p}(\overline{\mathfrak{t}})$ to the unique function mapping to it which is orthogonal to the kernel of $\mathcal{D}^* \mathcal{D}$ satisfies the required bound. On $\tau_{r,p}(\overline{\mathfrak{t}})$, $Q$ sends a function $\tau_{r,p} (f)$ to the corresponding function in $f \in \overline{\mathfrak{t}}$. Since the map $r^{-1} \tau_{r,p}$ is $\textnormal{Id} + O(r^{-1})$, so is its inverse, so on this factor, $Q$ satisfies a better bound than required. 
\end{proof}

Combining these results, we obtain a bound on the right inverse of the extremal operator.

\begin{proposition}\label{prop:inv-bounds} Fix a positive integer $p$ and let $\omega = \omega_{r,p}$. Denote \begin{align*} & G_{r,p} : (L_{k+4}^2)^T \times \overline{\mathfrak{t}}  \to (L_k^2)^T, \\ & G_{r,p}(\phi,h) = -\mathcal{D}^*_{\omega} \mathcal{D}_{\omega} (\phi) + \frac{1}{2} \langle \nabla \big( S(\omega) - \tau_{r,p}(h) \big), \nabla \phi \rangle -\tau_{r,p}(h).\end{align*} Then there exists a $C$ independent of $r$ such that $G_{r,p}$ has a right inverse $P_{r,p}$ with $\|P_{r,p}\|_{op} \leq C r^3$.

\end{proposition}

\begin{proof} Apply the Lemmas above and note that $G_{r,p}$ is  $- (\mathcal{D}^*_{\omega} \mathcal{D}_{\omega} + \tau_{r,p}(h) )$ plus decaying terms which decay faster than $(\mathcal{D}^*_{\omega} \mathcal{D}_{\omega} - \tau_{r,p}(h) )$. Thus for $r$ sufficiently large, $G_{r,p}$  also admits a right invertible and satisfies a similar bound, with a possibly larger constant $C$.  
\end{proof}

We next bound the nonlinear terms, analogously to Lemma 7.1 of \cite{fine1}. The operator we will use is the operator $F = F_{r,p}$ defined by
\begin{align}\label{operator}F(\psi, h)=S(\omega_{r,p} + \ddb \psi) -  \frac{1}{2}\langle \nabla_{\omega_r} \beta_p,&\nabla_{\omega_r} (\phi_{r,p})\rangle- \\ \nonumber   - \beta_p  - &\frac{1}{2} \langle \nabla(\tau_{r,p} (h)) , \nabla \psi \rangle -  \tau_{r,p} (h).\end{align}
Note that a zero of $F$ precisely gives a K\"ahler potential for an extremal metric. Moreover, the linearisation of $F$ is given by the operator $G_{r,p}$ of Proposition \ref{prop:inv-bounds}. The terms involving $\beta_p = \sum_{i=0}^p b_i r^{-i}$ are terms independent of $\psi$, added so that $F(0,0)$ is an approximate solution to the extremal equation. We also remark that $\beta_p$ is a potential with respect to $\omega_{r,0}$ of a holomorphic vector field on $X$, and so the $b_i$ depend on $r$. However, they are of the form $r^{-1} \tau_{r,0} (\tilde{b}_i) = \tilde{b}_i + O(r^{-1})$ for some \textit{fixed} $\tilde{b}_i \in \overline{\mathfrak{t}}$.

From now on we take the number of weak derivatives $k$ to be large enough such that $L^2_{k+4}$ in particular embeds into $C^{4,\alpha}$. This ensures that a solution to the extremal equation \eqref{operator} in $L^2_{k+4}$ in fact is smooth. It is also assumed in order to apply certain estimates from \cite{fine1}.

\begin{lemma}\label{lemma:nonlinearterms} Let $\mathcal{N}_{r,p} = F_{r,p} - G_{r,p} $. Then there are constants $c,C >0$ such that for all sufficiently large $r$, if $\phi,\psi \in (L^2_{k+4})^T$ satisfy $\| \phi \|_{L^2_{k+4}} , \| \psi \|_{L^2_{k+4}} \leq c$ then $$ \| \mathcal{N}_{r,p} (\phi) - \mathcal{N}_{r,p} (\psi) \|_{L^2_k } \leq C \big( \| \phi \|_{L^2_{k+4}} + \| \psi \|_{L^2_{k+4}} \big) \| \phi - \psi \|_{L^2_{k+4}}.$$

\end{lemma}

\begin{proof} This follows from the Mean Value Theorem. Indeed, let $\chi_t = (1-t) \phi + t \psi$. Then there is a $t\in [0,1]$ such that $$ \mathcal{N}_{r,p} (\phi) - \mathcal{N}_{r,p} (\psi) = \big(D \mathcal{N}_{r,p}\big)_{\chi_t} \big( \psi - \phi \big) . $$ Moreover, $\big(D \mathcal{N}_{r,p}\big)_{\chi_t}$ is the difference between the linearisations of the extremal map at $\omega_{r,p}$ and $\omega_{r,p} + \ddb \chi_t. $ So the estimate reduces to a similar estimate for the change in the linearised operator for a small change in the potential. This follows as in Fine's work from \cite[Lemma 2.10]{fine1}, which can be applied provided $k$ is chosen sufficiently large.
\end{proof}

These are the necessary ingredients to prove our main result.

\begin{proof}[Proof of Theorem \ref{mainthm-sec6}] The proof is an application of the quantitative implicit function theorem stated as Theorem \ref{implicitfnthm} to the extremal operator $F(\phi,h)$ of equation \eqref{operator}. 

The first step is to find a Lipschitz constant for  the non-linear terms of this operator, namely $F_{r,p}-DF_{r,p} = \mathcal{N}_{r,p}$. From Lemma \ref{lemma:nonlinearterms}, we  obtain a $C> 0$ such that for all balls of sufficiently small radius $\lambda$, the non-linear term $\mathcal{N}_{r,p}$ is Lipschitz on the ball of radius $\lambda$ with Lipschitz constant $\lambda C$. Thus for $r$ sufficiently large, the radius $\delta'_r $ of the ball on which the non-linear part of $F$ is Lipschitz with constant $ \frac{1}{2 \| P_{r,p} \|} $ is bounded below by a constant multiple of $r^{-3}$. It follows that the corresponding $\delta_r$ in the statement of Theorem \ref{implicitfnthm} is bounded below by $Cr^{-6}$, for some  (possibly different) constant $C>0$.

The next estimate that Theorem \ref{implicitfnthm} requires is a bound in $L^2_k$ of order $r^{-6-\varepsilon}$ for some $\varepsilon >0$ on the approximate solution $F(0,0)$. For this, we use the bounds of Theorem \ref{thm:approx-soln} on the metrics $\omega_{r,p}$. These estimates are pointwise estimates. However, as in \cite[Lemma 5.6 and 5.7]{fine1}, this implies a similar bound in the $C^{k}$-norm, by using a patching argument that allows one to go from bounds in the local model to global bounds. So in the $C^k$-norms, $F(0,0)$ is $O(r^{-p-1})$. Note that here we think of the $b_i$ in the expansion of $\beta_p$ as fixed functions, not depending on $r$, since they are of the form $r^{-1} \tau_{r,p} (\tilde{b}_i)$ for fixed functions $\tilde{b}_i$, and $$r^{-1} \tau_{r,p} (\tilde{b}_i) = \tilde{b}_i + O(r^{-1}).$$  

We require similar bounds with respect to the $L^2_{k}$-norms. For this note that the $O(r^{-p-1})$ bounds on $F(0,0)$  imply that the $p^{\textnormal{th}}$ approximate solution is $O(r^{-p- \frac{1}{2}})$, just as in \cite[Lemma 5.7]{fine1}. In particular, if $p$ is chosen to be $6$ or more, $F(0,0)$ satisfies the required bound and the extremal equation has a solution for all $r$ sufficiently large.
\end{proof}

 \bibliography{fibrations}
\bibliographystyle{amsplain}

\end{document}